\documentclass[11pt]{article}
\usepackage[english]{babel}
\usepackage{amsmath,amsthm,amssymb}
\usepackage[all]{xy}
\usepackage{setspace}
\usepackage{color}
\definecolor{grau}{rgb}{0.3,0.3,0.3}
\usepackage[colorlinks, linkcolor=grau, citecolor=grau, urlcolor=grau]{hyperref}
\usepackage{breakurl}
\usepackage{bibspacing}
\usepackage[top=1.3in, bottom=1.6in, left=1.3in, right=1.3in]{geometry}
\usepackage{booktabs}
\usepackage[ruled,vlined]{algorithm2e}
\usepackage{enumerate}

\newtheorem{theorem}{Theorem}[section]

\newtheorem{corollary}[theorem]{Corollary}
\newtheorem{lemma}[theorem]{Lemma}

\newtheorem{proposition}[theorem]{Proposition}

\theoremstyle{definition}

\newtheorem{remark}[theorem]{Remark}

\numberwithin{equation}{section}

\title{Integral points on Hilbert moduli schemes} 

\author{  Rafael von K\"anel  \  \textnormal{ and } \   Arno Kret
}
\newcommand{\OL}{\mathcal O}
\newcommand{\iS}{J}

\newcommand{\Pic}{\textup{Pic}}
\newcommand{\uM}{\textup{M}}

\newcommand{\isomto}{\overset \sim \to}
\newcommand{\GL}{{\textup {GL}}}

\newcommand{\idf}{{\mathfrak f}} 
\newcommand{\ZZ}{\mathbb Z}
\newcommand{\Z}{\mathbb Z}
\newcommand{\Aut}{\textup{Aut}}

\newcommand{\Q}{\mathbb Q}
\newcommand{\R}{\mathbb R}

\newcommand{\Hom}{\textup{Hom}}

\newcommand{\spec}{\textnormal{Spec}}

\newcommand{\QQ}{\mathbb Q}

\newcommand{\End}{\textup{End}}

\newcommand{\ces}{(ES)}

\newcommand{\absg}{\underline{A}_g}
\newcommand{\hilbmod}{M}
\newcommand{\abomult}{\underline{M}}
\newcommand{\order}{\Gamma}
\newcommand{\sch}{\textnormal{(Sch)}}
\newcommand{\sets}{\textnormal{(Sets)}}
\newcommand{\SL}{\textnormal{SL}}

\newcommand{\CC}{\mathbb C}

\newcommand{\RR}{\mathbb R}

\definecolor{darkgreen}{rgb}{0.0, 0.4, 0.26}

\begin{document}

\maketitle

\vspace{0.5cm}

\begin{abstract}
We use the method of Faltings (Arakelov, Par\v{s}in, Szpiro) in order to explicitly study integral points on a class of varieties over $\ZZ$ called Hilbert moduli schemes. For instance, integral models of Hilbert modular varieties are classical examples of Hilbert moduli schemes. Our main result gives explicit upper bounds for the height and the number of integral points on Hilbert moduli schemes.
\end{abstract}

\vspace{1cm}

\tableofcontents
\vspace{0.5cm}

\section{Introduction}

In this paper we explicitly study integral points on a class of varieties over $\ZZ$ called Hilbert moduli schemes. We now to try to briefly explain and motivate the notion of Hilbert moduli schemes. Let $L/\QQ$ be a totally real number field of degree $g$ with ring of integers $\OL$.


\paragraph{Hilbert moduli schemes.}  Intuitively one can think of a Hilbert moduli scheme as a variety whose points parametrize pairs $(A,\alpha)$ where $A$ is a `polarized' abelian variety of dimension $g$ with $\OL$-multiplication and $\alpha$ lies in a certain set $\mathcal P(A)$ associated to $A$. For example, if we take the set $\mathcal P(n)(A)=\{(\OL/n\OL)^2\cong A_n\}$ of principal level $n$-structures on $A$ for some $n\in\ZZ_{\geq 3}$, then there is a quasi-projective scheme $Y(n)$ over $\ZZ[1/n]$ parametrizing $(A,\alpha)$ with $\alpha\in \mathcal P(n)(A)$. Moreover a connected component of $Y(n)(\CC)$ identifies with $\mathbb H^g/\Gamma(n)$, where $\Gamma(n)=\ker(\SL_2(\OL)\to \SL_2(\OL/n\OL))$ acts on $\mathbb H^g$ via the $g$ distinct real embeddings of $L$ and the action of $\SL_2(\RR)$ on $\mathbb H=\{z\in\CC; \, \textnormal{im}(z)>0\}$. 
There are many interesting examples of varieties in the literature which are (birationally equivalent to) Hilbert moduli schemes: For instance \cite{hiva:hmsclass,hiza:hmsclass,vandergeer:hilbertmodular} contain many surfaces of general type and also surfaces which are rational, elliptic over $\mathbb P^1$, or blown-up K3. 

Further, it was shown in \cite{rvk:intpointsmodellhms} for $g=1$  that the general moduli formalism allows to explicitly study  a class of Diophantine equations of interest which is a priori substantially more general than just defining equations\footnote{Here by defining equations we mean equations which define a variety over $\QQ$ whose set of complex points has a connected component which identifies with $\mathbb H^g/\Gamma$.} of $\mathbb H^g/\Gamma$ for certain subgroups $\Gamma$ of $\SL_2(\OL)$. This motivates to work for any $g\geq 1$ with the following general notion: Let $\mathcal M$ be the Hilbert moduli stack associated to $L$ by Rapoport~\cite{rapoport:hilbertmodular} and Deligne--Pappas~\cite{depa:hilbertmodular}, see Section~\ref{sec:hms}. We say that a scheme $Y$ is a Hilbert moduli scheme of a presheaf $\mathcal P$ on $\mathcal M$, and we write $Y=M_{\mathcal P}$, if $\mathcal P$ is representable by an object in $\mathcal M(Y)$. In the case $L=\QQ$, the Hilbert moduli stack $\mathcal M$ identifies with the moduli stack $\mathcal M_{1,1}$ of elliptic curves and the Hilbert moduli schemes are precisely the moduli schemes of elliptic curves in Katz--Mazur~\cite[4.3.1]{kama:moduli}. For example, if $Y\hookrightarrow \mathbb A^2_\ZZ$ is defined by $y^2=x^3+a$ with $a\in\ZZ$ nonzero, then $Y$ becomes over $\ZZ[1/\nu]$, $\nu=6a$, a Hilbert moduli scheme of a presheaf $\mathcal P_a$ on $\mathcal M_{1,1}$ such that  $|\mathcal P_a(A)|\leq 24$ for each $A\in\mathcal M_{1,1}(\CC)$; see \cite[$\mathsection 3.2.2$]{rvk:intpointsmodellhms}.



\subsection{Results} We continue our notation. Let $Y$ be a variety\footnote{A variety $Y$ over a ring $R$ is an $R$-scheme which is separated and of finite type.} over $\ZZ$, and let $\ZZ_S\supseteq \ZZ$ be the ring of $S$-integers in $\QQ$ where $S$ is a finite set of rational primes. 
Faltings~\cite{faltings:finiteness} gives  the following  finiteness result (see Section~\ref{sec:generalfin}): The set $Y(\ZZ_S)$ is finite if $Y$ becomes over a ring $\ZZ[1/\nu]$, $\nu\in\ZZ_{\geq 1}$,  a Hilbert moduli scheme of a presheaf $\mathcal P$ on $\mathcal M$ with $|\mathcal P|_{\CC}<\infty$, where $$|\mathcal P|_{\CC}=\sup_{A\in\mathcal M(\CC)}|\mathcal P(A)|$$ is the maximal number of $\mathcal P$-level structures over $\CC$. For those $\mathcal P$ of interest in arithmetic, it is usually not difficult to compute $\nu=\nu(Y)$ and to show $|\mathcal P|_{\CC}<\infty$. For instance it holds $|\mathcal P(n)|_{\CC}\leq n^{4g}$ and $|\mathcal P_a|_{\CC}\leq 24$, and we can take $\nu=n$  for $\mathcal P(n)$ and $\nu=6a$ for $\mathcal P_a$. 

In this paper we prove a fully explicit version of the above finiteness result for $Y(\ZZ_S)$. More precisely, if we write $N_S=\prod_{p\in S}p$ where $N_S=1$ when $S$ is empty, then our main result (Theorem~\ref{thm:main}) is a slightly more general version of the following theorem.

\vspace{0.3cm}
\noindent{\bf Theorem A.}
\emph{Suppose that $Y$ is a variety over $\ZZ$, which becomes over a ring $\ZZ[1/\nu]$, $\nu\in\ZZ_{\geq 1}$, a Hilbert moduli scheme of some presheaf $\mathcal P$ on $\mathcal M$. Then the following holds.
\begin{itemize}
\item[(i)] Any point $P\in Y(\ZZ_S)$ satisfies $h_\phi(P)\leq (3g)^{144g}(\nu N_S)^{24}.$ 
\item[(ii)] If $e_g=(8g)^{8}$ then $|Y(\ZZ_S)|\leq |\textnormal{Pic}(\OL)||\mathcal P|_\CC(2\nu N_S)^{e_g}.$ 
\end{itemize}}
\vspace{0.1cm}
\noindent Here $h_{\phi}:Y(\bar{\QQ})\to \RR$ is a height, defined in \eqref{def:height}, which depends on the choice of $\mathcal P$ and which has the Northcott property if $|\mathcal P|_\CC<\infty$, see Section~\ref{sec:hms2}. However the  bound for $h_\phi$ in (i) can also be useful when $|\mathcal P|_\CC=\infty$: For example cubic Thue--Mahler equations define   moduli schemes $Y=M_{\mathcal P}$ with $|\mathcal P|_\CC=\infty$, see \cite[p.60]{vkma:computation}. We remark that the case $g=1$ of Theorem~A was established in \cite[Thm 7.1]{rvk:intpointsmodellhms}. Further, we tried to simplify the form of our bounds. In fact the exponents in Theorem~A can be improved up to a certain extent, especially when one restricts to specific moduli schemes. 

\paragraph{Applications.}While there are many explicit finiteness results for integral points on curves, much less is known for higher dimensional varieties. In light of this, an interesting aspect of Theorem~A is that it can be applied to explicitly study integral points on certain classical varieties for which no effective method is known. 
To illustrate this, we take an ideal $\mathfrak n\subseteq \OL$ of norm $n\geq 1$ and  we consider the presheaf $\mathcal P_1(\mathfrak n)$ on $\mathcal M$ defined in Section~\ref{sec:p1n} following \cite{kama:moduli,pappas:hilbmod}: It parametrizes $\mathfrak n$-torsion points of `exact\footnote{An $\mathfrak n$-torsion point $P$ has exact order $\mathfrak n$ if its annihilator $\textnormal{Ann}_\OL (P)$ equals $\mathfrak n$.} order $\mathfrak n$',  and for $g=1$ it is the $\Gamma_1(n)$-moduli problem of \cite{kama:moduli}.  Suppose now that $\mathcal P_1(\mathfrak n)$ is representable over $\ZZ[1/n]$ with Hilbert moduli scheme $Y_1(\mathfrak n)=M_{\mathcal P_1(\mathfrak n)}$ a (quasi-projective) variety over $\ZZ[1/n]$ of relative dimension $g$. Theorem~A and $|\mathcal P_1(\mathfrak n)|_\CC\leq n^{2g}$ then lead to the following:


\vspace{0.3cm}
\noindent{\bf Corollary.}
\emph{Let $Y$ be a variety over $\ZZ$, which becomes over $\ZZ[1/n]$ isomorphic to $Y_1(\mathfrak n)$.
\begin{itemize}
\item[(i)] Any point $P\in Y(\ZZ_S)$ satisfies $h_\phi(P)\leq (3g)^{144g}(nN_S)^{24}$.
\item[(ii)] The cardinality of $Y(\ZZ_S)$ is at most $|\textnormal{Pic}(\OL)|(2nN_S)^{e_g}.$
\end{itemize} }
\vspace{0.1cm}

\noindent  As far as we know, this is the first explicit finiteness result for $Y_1(\mathfrak n)(\ZZ_S)$ when $Y_1(\mathfrak n)$ is a surface or has higher dimension. However when $Y_1(\mathfrak n)$ is a curve, other methods produced much stronger results (write $n=\mathfrak n$):  For example, if $n\geq 4$ then $\mathcal P_1(n)$ is representable over $\ZZ[1/n]$ and Bilu \cite{bilu:israel,bilu:modular} and Sha~\cite{sha:intmodimrn,sha:intmod2} obtained strong effective height bounds for the set $j^{-1}(\ZZ_S)\subseteq Y_1(n)(\QQ)$ which contains $Y_1(n)(\ZZ_S)$ for $j:Y_1(n)\to \mathbb A^1_{\ZZ[1/n]}$ the $j$-map. Their results hold moreover for points defined over arbitrary number fields $K$ and for quite general classes of modular curves. Furthermore, $Y_1(n)$ has no $\QQ$-rational point when $n=11$ or $n\geq 13$ by Mazur~\cite{mazur:eisenstein}.
Theorem~A holds only for $K=\QQ$, and a priori only for  $S$-integral points. However, it implies the corollary  more generally (\cite{vkkr:repcond}) for integral models $Y_\Gamma$ of any representable Hilbert modular variety of dimension $g$ associated to a congruence subgroup $\Gamma\subset \textnormal{SL}_2(\OL)$. This is particularly interesting for $g\geq 2$, 
since it is notoriously difficult to obtain effective finiteness for integral points on higher dimensional varieties; see for example Levin~\cite{levin:intpointsrunge,levin:intpointslogforms} and Le Fourn~\cite{lefourn:A2rungeparis,lefourn:A2runge,lefourn:tubularbaker}. 

Moreover, Theorem~A allows to explicitly study Diophantine equations which are not necessarily defining equations of $\mathbb H^g/\Gamma$. For example, it was shown in \cite{rvk:intpointsmodellhms} that the case $g=1$ of (i) gives explicit Weil height bounds for the solutions of several classical Diophantine equations defining moduli schemes of elliptic curves: Besides Mordell equations $y^2=x^3+a$ via $\mathcal P_a$ (\cite[Cor 7.4]{rvk:intpointsmodellhms}) and related equations such as cubic Thue, cubic Thue--Mahler and generalized Ramanujan--Nagell equations (K.--Matschke~\cite[$\mathsection$8,9]{vkma:computation}), this also works for $S$-unit equations (\cite[Cor 7.2]{rvk:intpointsmodellhms} and independently  Murty--Pasten~\cite[Thm 1.1]{mupa:modular}; see also Frey~\cite[p.544]{frey:ternary}). Furthermore, it was demonstrated in \cite{vkma:computation} that these explicit Weil height bounds combined with efficient sieves constructed in \cite{vkma:computation} allow to solve these classical equations in practice. We are currently trying to work out similar explicit applications of Theorem~A for equations defining Hilbert moduli schemes with $g\geq 2$. To this end, if $g$ is not too large, say $g\leq 100$, then our simplified height bounds in (i) are in fact already sufficiently strong for practical computations when combined with efficient sieves which usually can deal in practice with huge initial bounds.


\subsection{Idea of proofs}

We continue our notation. To prove our results, we use and generalize the strategy of \cite[Thm 7.1]{rvk:intpointsmodellhms} in which Theorem~A was obtained for $g=1$ and we apply the method of Faltings~\cite{faltings:finiteness} (Arakelov, Par\v{s}in, Szpiro) as follows. As in Theorem~A, let $Y$ be a variety over $\ZZ$ with $Y_{\ZZ[1/\nu]}=M_\mathcal P$  a Hilbert moduli scheme of some presheaf $\mathcal P$ on $\mathcal M$. We first use the moduli formalism to obtain a natural map induced by forgetting extra structures, 
$$\phi:Y(\ZZ_S)\to \absg(T),$$
which we call Par\v{s}in construction. Here $\absg(T)$ is the set of isomorphism classes of abelian schemes over $T=\spec(\ZZ_S[1/\nu])$ of relative dimension $g$; this set is finite by \cite{faltings:finiteness,zarhin:shafarevich}. Moreover, the effective Shafarevich conjecture  (see Section~\ref{sec:es}) allows to explicitly control $\absg(T)$ and  then it suffices to study the fibers of $\phi$ since $Y(\ZZ_S)=\phi^{-1}(\absg(T))$. While proving the effective Shafarevich conjecture is currently out of reach, this conjecture is known in important cases. For example, the case of abelian schemes of product $\GL_2$-type was established in \cite{rvk:modularhms} via Faltings method~\cite{faltings:finiteness} and Serre's modularity conjecture \cite{khwi:serre}, and this case is sufficient for our purpose. Indeed it turns out that
\begin{equation}\label{eq:gl2factorintro}
\phi:Y(\ZZ_S)\to M_{\GL_{2,g}}(T)\hookrightarrow \absg(T)
\end{equation}
factors through the subset $M_{\GL_2,g}(T)$ of all $A\in\absg(T)$ such that $\End(A)\otimes_\ZZ\QQ$ has a commutative semi-simple $\QQ$-subalgebra of degree $g$. Then the bound for $h_\phi(P)$ in Theorem~A~(i)  follows from the explicit bound in \cite{rvk:modularhms} for the stable Faltings height $h_F$ on $M_{\GL_2,g}(T)$, since $h_\phi=\phi^*h_F$. Further, we show that the height $h_\phi$ on $Y(\bar{\QQ})$ has the Northcott property if the presheaf $\mathcal P$ is finite over $\mathcal M(\bar{\QQ})$. Here the proof  combines the Northcott property of $h_F$ in \cite{fach:deg} with our decomposition (described below) of $\phi$.

To prove the explicit bound for $|Y(\ZZ_S)|$ in Theorem A (ii), we apply the bound for $|M_{\GL_2,g}(T)|$ in \cite{rvk:modularhms}. This reduces the problem to suitably controlling $\deg(\phi)$, where for any map $f:X\to Z$ of sets we write $\deg(f)=\sup_{z\in Z}|f^{-1}(z)|$, since \eqref{eq:gl2factorintro}  gives 
$$|Y(\ZZ_S)|\leq \deg (\phi) |M_{\GL_2,g}(T)|.$$  
The obvious approach to study $\phi$ would be to factor it through the finite map $\phi_{\lambda}:A_{g}\to \absg$ induced by forgetting polarizations, where $A_g=A_{g,1}$ for $A_{g,d}$ defined in Section~\ref{sec:parsin}. However, this approach is problematic for two reasons. Firstly, in general it is not clear how to map in a controlled way the set $Y(\ZZ_S)$ to $A_g(T)$ nor to some $A_{g,d}(T)$ with $d\geq 2$. Secondly, the forget polarization map $\phi_\lambda$ of $A_{g,d}$ has very complicated fibers. To circumvent these difficulties,  we develop a different approach to  study $\phi$ and we decompose $\phi$ as
\begin{equation}\label{eq:decompintro}
\phi:Y(\ZZ_S)\to^{\phi_\alpha} M(T)\to^{\phi_\varphi} \abomult(T)\to^{\phi_\iota} M_{\GL_{2,g}}(T).
\end{equation} 
Here $M(T)$ (resp. $\abomult(T)$) is the set of isomorphism classes of triples $(A,\iota,\varphi)$ (resp. pairs $(A,\iota)$) where $A\in \absg(T)$, $\iota:\OL\to \End(A)$ is a ring morphism and $\varphi$ is a `polarization' of $(A,\iota)$.  Now, the crucial advantage of \eqref{eq:decompintro} is that the natural map $\phi_\varphi:\hilbmod(T)\to \abomult(T)$ induced by forgetting the polarization $\varphi$ has much simpler fibers than $\phi_\lambda:A_g(T)\to \absg(T)$. Indeed, since $\varphi$ is compatible with the $\OL$-action $\iota$ and $L=\OL\otimes_\ZZ\QQ$ has degree $g$, we can reduce to the commutative case and then Dirichlet's unit theorem for orders in $L$ leads to 
\begin{equation*}
\deg(\phi_\varphi)\leq 2^g.
\end{equation*}
Here the reduction, done in Lemmas~\ref{lem:gl2avstructure}, \ref{lem:polformal} and \ref{lem:psipolbound}, consists mostly of formal computations but it also involves a Lie algebra argument requiring that the function field $k(T)=\QQ$. The map $\phi_\alpha$, defined in \eqref{def:forgetlvl}, is induced by forgetting $\mathcal P$-structures. For most $\mathcal P$ of interest in arithmetic, the fibers of $\phi_\alpha$ encapsulate deep arithmetic information and it is currently out of reach to explicitly compute these fibers. However, we only need to bound their size and the formal arguments in Lemmas~\ref{lem:degforgetlvl} and \ref{lem:geomlvlred}, using that $\mathcal M/\ZZ$ is separated, give  
\begin{equation*}
\deg(\varphi_\alpha)\leq |\mathcal P|_\CC.
\end{equation*} 
A disadvantage of our approach via the decomposition \eqref{eq:decompintro} is that we have to deal with the map $\phi_\iota$ induced by forgetting $\iota$. This map is finite, but it has very complicated fibers since $\iota$ is not anymore compatible with a fixed polarization. In fact for most base schemes $T$ it is currently impossible to obtain an explicit bound for $\deg(\phi_\iota)$ in terms of $T$ and $g$. However, for our $T$ with $k(T)=\QQ$, Theorem~\ref{thm:endobound} implies such a bound which then combined with the above displayed results proves Theorem~A (ii). 

We next discuss Theorem~\ref{thm:endobound}. An $\OL$-structure on an abelian scheme $A$ is an $\OL$-module structure on $A$ up to $\Aut(A)$-conjugation, that is a ring morphism $\OL\to R$ modulo conjugation action of $R^\times$ on $R=\End(A)$.   
On taking $\Gamma=\OL$ in Theorem~\ref{thm:endobound}, which holds more generally for any order $\Gamma$ of a general number field, we obtain the following result. 

\vspace{0.3cm}
\noindent{\bf Theorem B.}
\emph{For any finite set $S$ of rational primes, there exist at most $|\textnormal{Pic}(\OL)|(2N_S)^{e_g}$  distinct $\OL$-structures on any abelian scheme over $\ZZ_S$ of relative dimension $g$.}
\vspace{0.3cm}

\noindent  The Jordan--Zassenhaus theorem (JZ) implies finiteness of $\OL$-structures on a fixed abelian scheme $A$ in much more general situations, see Lemma~\ref{lem:generalgammastr}. However this finiteness via (JZ) is ineffective, see Section~\ref{sec:generalgammastr}. In view of this we developed a different approach which exploits that $R=\End(A)$ comes from geometry. 
The principal ideas are as follows: For any abelian scheme $A$ as in Theorem~B, we first make a reduction to the key case when $R=\uM_n(\OL_K)$ for $\OL_K$ the ring of integers of a subfield $K\subseteq L$ of relative degree $n$. This reduction consists of several steps and is quite involved. For example it requires to replace $\OL$ by the order $\Gamma=\ZZ[\delta\OL]$ of $L$, since we apply a carefully chosen isogeny $A\to B^n$ of degree $\delta$ where $\End(B)=\OL_K$. Here we can control the degree $\delta$ in terms of $g$ and $N_S$ by the uniform isogeny estimate in \cite{rvk:modularhms}, which in turn relies on the most recent version, due to Gaudron--R\'emond~\cite{gare:isogenies}, of the Masser--W\"ustholz isogeny theorem~\cite{mawu:abelianisogenies} based on transcendence. To prove the result in the key case, we decompose the set of $\Gamma$-structures on $R=\textnormal{M}_n(\OL_K)$ into $|\Hom(K,L)|$ distinct subsets $\Sigma_\varphi$ of $\varphi$-compatible $\order$-structures on $R$ for $\varphi:K\to L$. Then we identify $\Sigma_\varphi$ with a set $J_\varphi$ of isomorphism classes of certain $\Gamma$-modules and we construct an embedding of $J_\varphi$ into the monoid $C_\Gamma$ of equivalence classes of (not necessarily invertible) fractional ideals of $\Gamma$. To bound $|C_\Gamma|$, we apply an idea of Lenstra which computes $C_\Gamma=\textnormal{Pic}(\Gamma)\cdot \mathcal I$ with $\mathcal I$ a controlled finite set and we use algebraic number theory to relate $|\Pic(\Gamma)|$ with $|\Pic(\OL)|$. This then leads to Theorem~B.

Finally, we point out that the above described proofs crucially exploit (at several places) that the ground field is $\QQ$ and it is clear that substantially new ideas are required to generalize our explicit results in Theorem~A to arbitrary number fields.



\subsection{Organization of the paper}

\paragraph{Outline of the paper.} After recalling  basic properties of orders and abelian schemes in Section~\ref{sec:orders}, we review in Section~\ref{sec:hms} the construction of the Hilbert modular stack following \cite{depa:hilbertmodular}. Here we also define Hilbert moduli schemes and we discuss basic properties of a height on such schemes. Then we state our main result in Section~\ref{sec:results}. 

In Sections~\ref{sec:endopreliminaries}-\ref{sec:polarizations} we prove various statements which are used in the proofs of our theorems. First, we collect in Section~\ref{sec:endopreliminaries} some preliminary results for  endomorphisms of abelian schemes over Dedekind schemes. In particular, we review here a variant of the (Serre) tensor product of abelian schemes with certain not necessarily projective modules and we consider in more detail the endomorphism ring of an abelian scheme of $\GL_2$-type. In Section~\ref{sec:parsin} we use the natural forgetful map of a Hilbert moduli scheme as a Par\v{s}in construction and we decompose this forgetful map in a way which is useful for the explicit study of its fibers. In Section~\ref{sec:es} we  collect some known results concerning the effective Shafarevich conjecture, while in Section~\ref{sec:levelstructures} we prove useful formal properties of level structures. Then we study in Section~\ref{sec:polarizations} polarizations of abelian schemes with $\OL$-multiplication and we explicitly control the number of such polarizations over certain base schemes. 

In Section~\ref{sec:endo} we discuss and prove Theorem~\ref{thm:endobound}, which implies Theorem~B providing an explicit bound for the number of $\OL$-structures on certain abelian schemes. The proof of Theorem~\ref{thm:endobound} uses the strategy outlined above: In Section~\ref{sec:endokeycase} we first establish a sharper version of Theorem~\ref{thm:endobound} in the key case. Then  we show in Section~\ref{sec:endoreduction} how to reduce the problem to the key case, and finally we prove the theorem by putting everything together in Section~\ref{sec:endoproof}. We also discuss in Section~\ref{sec:generalgammastr} a more general finiteness result for $\OL$-structures, which we deduce from the Jordan--Zassenhaus theorem. 

Finally, in Section~\ref{sec:proofs} we  combine the results obtained in the previous sections in order to prove Theorem~\ref{thm:main}. As a byproduct we also obtain a proof of Proposition~\ref{prop:northcott} on the Northcott property. Further, we show in Section~\ref{sec:generalfin} how to modify our strategy in order to deduce in Proposition~\ref{prop:fin} a more general finiteness result for $S$-integral points on Hilbert moduli schemes. Here the main ingredient is the unpolarized Shafarevich conjecture for abelian varieties over number fields proven by Faltings~\cite{faltings:finiteness} and Zarhin~\cite{zarhin:shafarevich}.

\paragraph{Conventions and notation.} Unless mentioned otherwise, we shall use throughout this paper the following conventions and notation. We denote by $\sets$ and $\sch$ the categories of sets and schemes respectively. Let $S$ be a scheme and let $\mathcal C$ be a category. As usual, we say that a contravariant functor from $\mathcal C$ to $\sets$ is a presheaf on $\mathcal C$  
and we write $\Hom_S$ for $\Hom_{\mathcal C}$  when $\mathcal C$ is the category of $S$-schemes. Further, we shall often omit the subscript $\mathcal C$ of $\Hom_{\mathcal C}$  when it is clear from the context in which category we are working. For example if $A$ and $B$ are abelian schemes over $S$, then $\Hom(A,B)$ denotes the set of $S$-group scheme morphisms $A\to B$ and $\Hom^0(A,B)=\Hom(A,B)\otimes_\ZZ\QQ$. 

If $T$ and $Y$ are $S$-schemes, then we define $Y(T)=\Hom_S(T,Y)$  and we write $Y_T=Y\times_S T$ for the base change of $Y$ from $S$ to $T$. We often identity an affine scheme $S=\spec(R)$ with the ring $R$. For example, if $T=\spec(R)$ is  affine then we write $Y_R$ for $Y_T$ and $Y(R)$ for $Y(T)$. Following \cite{bolura:neronmodels}, we say that $S$ is a Dedekind scheme if $S$ is a normal noetherian scheme of dimension 0 or 1. Further, a variety $Y$ over $S$ is an $S$-scheme $Y$ whose structure morphism $Y\to S$ is separated and of finite type.

For any set $M$, we denote by $\lvert M\rvert$ the number of distinct elements of $M$.  Let $f:M'\to M$ be a map of sets. Then we define $\deg(f) = \sup_{m \in M} |f^{-1}(m)|$, and we say that the map $f$ is finite if for each $m\in M$ the fiber $f^{-1}(m)$ of $f$ over $m$ is finite. By $\log$ we mean the principal value of the natural logarithm and we define the product taken over the empty set as $1$. Finally, for any field $k$ we denote by $\bar{k}$ an algebraic closure of $k$.

\paragraph{Acknowledgements.} We would like to thank Shouwu Zhang for useful comments. Also we are grateful to Hendrik Lenstra for answering our question about  $\Gamma$-fractional ideals for orders $\Gamma$ of number fields; his explanations and ideas resulted in Lemma~\ref{lem:classnumberbound}.

This paper is the first in a series of papers containing the results obtained in our long term project on integral points of certain higher dimensional varieties. While working on this project over the last years, we were supported by the IH\'ES, IAS Princeton, MSRI, MPIM Bonn, Princeton University, University of Amsterdam, and IAS Tsinghua university. We would like to thank all of these institutions for their support. Further we were supported by an EPDI fellowship, NSF grant $300.154591$, and Veni grant.

\section{Abelian schemes and orders}\label{sec:orders}
In this section we introduce some notation and terminology which will be used throughout the paper. Let $S$ be a scheme and let $\mathcal O$ be a not necessarily commutative ring.

\paragraph{Category of $\mathcal O$-abelian schemes.}  An $\mathcal O$-abelian scheme  over $S$ is a left $\mathcal O$-module object in the category of abelian schemes over $S$: The objects of the category of $\OL$-abelian schemes over $S$ are given by pairs $(A, \iota)$ consisting of an abelian scheme $A$ over $S$ and a ring morphism $\iota \colon \OL \to \End(A)$. A morphism $f \colon (A, \iota) \to (A', \iota')$ of $\OL$-abelian schemes over $S$ is a morphism $f \colon A \to A'$ of abelian schemes over $S$ such that $f\iota(x)=\iota'(x)f$ for all $x\in \OL$. In situations where the specific choice of $\iota$ is not relevant, we usually omit $\iota$ from the notation  and we simply say $A$ is an $\mathcal O$-abelian scheme over $S$.

\paragraph{Tensor products of abelian schemes.} Let $A$ be an $\mathcal O$-abelian scheme over $S$ and let $I$ be a finite projective right $\mathcal O$-module. To define the (Serre) tensor product $I\otimes_\mathcal O A$, we observe that $T\mapsto I\otimes_{\mathcal O}A(T)$ defines a contravariant functor from the category of $S$-schemes to the category of abelian groups. This functor is represented by an abelian scheme $I\otimes_{\mathcal O}A$ over $S$, see for example \cite[Thm 7.2 and Thm 7.5]{conrad:gz}. In the case of an arbitrary base scheme $S$, the assumption that $I$ is projective is necessary to assure representability  by an abelian scheme over $S$. 
However in certain situations of interest one can  remove (see Section~\ref{sec:tensorprod}) the assumption that $I$ is projective. 

\paragraph{Orders.} 
Let $\Omega$ be a finite $\QQ$-algebra which is not necessarily commutative.  We say that $\mathcal O$ is an order  of $\Omega$ if it is a finite $\ZZ$-subalgebra of $\Omega$ and $\QQ\OL=\Omega$.  
Then $\OL$ is an order of $\Omega$ if and only if $\OL$ is a $\ZZ$-order in $\Omega$ in the sense of Reiner~\cite{reiner:orders}.

In this paper we mostly work with two classes of orders. To discuss the first class, let $K$ be a number field. Then our definition of an order $\OL$ of $K$ is equivalent to the usual definition \cite[Def 12.1]{neukirch:ant} saying that $\OL$ is a subring of the ring of integers $\OL_K$ of $K$ such that $\OL$ contains an integral basis of length $[K:\QQ]$. Indeed if $\OL$ is an order of $K$, then $\OL$ is a subring of its normalization in $K$ which is $\OL_K$ 
and thus any basis of the free $\ZZ$-module $\OL$ is an integral basis of the $\QQ$-vector space $K=\QQ\OL$ which automatically has length $[K:\QQ]$. 
 Another interesting class of orders is given by the endomorphism rings of abelian varieties: If $A$ is an abelian variety over an arbitrary field, then the endomorphism ring $\End(A)$ of $A$ is a finitely generated torsion-free abelian group  and thus $\End(A)$ identifies with an order of the $\QQ$-algebra $\End^0(A)$. Here $\End(A)$ is not necessarily commutative. 

We say that $\mathcal O$ is a maximal order of $\Omega$ if it is an order of $\Omega$ which is not strictly contained in any other order of $\Omega$. 
If our finite $\QQ$-algebra $\Omega$ is in addition semi-simple, then $\Omega$ is a separable $\QQ$-algebra in the sense of \cite{reiner:orders} and therefore \cite[Cor 10.4]{reiner:orders} gives the following: Any order of $\Omega$ is contained in a maximal order of $\Omega$, and there exists at least one maximal order of $\Omega$. 
In particular $\End(A)$ is always contained in a maximal order of $\End^0(A)$, since the finite $\QQ$-algebra $\End^0(A)$ is semi-simple.

\paragraph{Dual abelian scheme and polarizations.}Let $A$ be an abelian scheme over  $S$.  Its dual $A^\vee=\Pic^0(A)$ is also an abelian scheme over $S$, see for example \cite[Chap I]{fach:deg}. A morphism $\varphi:A\to A^\vee$ of abelian schemes over $S$ is called symmetric if $\varphi$ equals the composition of the dual $\varphi^\vee=\Pic^0(\varphi)$  with the canonical isomorphism $A\isomto (A^\vee)^\vee$. We denote by $\Hom(A,A^\vee)^\textnormal{sym}$ the set of symmetric morphisms from $A\to A^\vee$. A symmetric morphism $\varphi:A\to A^\vee$ is called a polarization of $A$ if the line bundle $(1,\varphi)^*\mathcal P_A$  is relatively ample on $A$.  Here $(1,\varphi)^*\mathcal P_A$ is the pullback by $(1,\varphi):A\to A\times_S A^\vee$ of the Poincar\'e line bundle $\mathcal P_A$ on $A\times_S A^\vee$. 
A principal polarization of $A$ is a polarization of $A$ of degree one. We denote by $\textnormal{Pol}(A)$ the set of polarizations of $A$. If $(A,\iota)$ is an $\OL$-abelian scheme over $S$, then sending $x\in\OL$ to the endomorphism $\Pic^0(\iota(x))$ of $A^\vee$ defines a ring morphism $\iota^\vee:\OL\to \End(A^\vee)$ and therefore $(A^\vee,\iota^\vee)$ is an $\OL$-abelian scheme over $S$. 

\paragraph{Dedekind base scheme.} Let $S$ be a connected Dedekind scheme. Any abelian scheme over $S$ is the N\'eron model of its generic fiber. Thus, if $A$ and $A'$ are abelian schemes over $S$, then base change from $S$ to the function field $k$ of $S$ induces canonical isomorphisms 
\begin{equation}\label{eq:neron}
\Hom(A,A')\isomto \Hom(A_k,A_{k}') \quad \textnormal{and} \quad \End(A)\isomto \End(A_k). 
\end{equation}
This is a formal consequence (see for example \cite[$\mathsection$9.3.1]{rvk:modularhms}) of the universal property of the N\'eron model. In what follows we will often exploit \eqref{eq:neron}  to reduce  statements over any connected Dedekind scheme $S$ to the case when $S$ is the spectrum of a field.

\section{Hilbert moduli schemes}\label{sec:hms}


\noindent In the first part of this section we review the Hilbert moduli stack  constructed by Rapoport~\cite{rapoport:hilbertmodular} and Deligne--Pappas~\cite{depa:hilbertmodular}. In the second part we define Hilbert moduli schemes and we discuss their natural forgetful maps. We also define a height on Hilbert moduli schemes by generalizing the construction of \cite{rvk:intpointsmodellhms}.

\subsection{The Hilbert moduli stack}

Let $L$ be a totally real number field of degree $g$ with ring of integers $\OL$, and let $I$ be a nonzero finitely generated $\OL$-submodule of $L$.                           For each real embedding $\sigma$ of $L$ we choose an orientation on the real line $I\otimes_{\OL,\sigma}\R$, that is a connected component $I_{+}^\sigma$ of this real line minus zero. We denote by $I_+$  the set of totally positive elements of $I$, where $\lambda\in I$ is totally positive if for each $\sigma$ the image of $\lambda$ lies in $I_+^\sigma$.   Following Deligne--Pappas~\cite{depa:hilbertmodular,pappas:hilbmod}, we equip the category $\sch$ with the \'etale topology and we denote by $\mathcal M^I$ the stack over $\sch$  whose objects over an arbitrary scheme  $S$ are given by triples $x=(A,\iota,\varphi)$:
\begin{itemize}
\item[(i)] An abelian scheme $A$ over $S$ of relative dimension $g$.
\item[(ii)] A ring morphism $\iota:\OL\to \End(A)$.
\item[(iii)] A morphism $\varphi:I\to \Hom_\OL(A,A^\vee)^\textnormal{sym}$ of $\OL$-modules such that the induced morphism $I\otimes_{\mathcal O}A \to A^\vee$ is an isomorphism and such that $\varphi(I_+)\subseteq \textnormal{Pol}(A)$.
\end{itemize}
Here $\Hom_\OL(A,A^\vee)^\textnormal{sym}$ denotes the $\OL$-module of symmetric morphisms $(A,\iota)\to(A^\vee,\iota^\vee)$, and $I\otimes_\OL A$ is the abelian scheme over $S$ which represents the tensor product  of $A$ with the finite projective $\OL$-module $I$ (see Section~\ref{sec:orders}). The fiber $\mathcal M^I(S)$ of $\mathcal M^I$ over $S$ is a groupoid. A morphism $f:x\to x'$ of objects of  $\mathcal M^I(S)$ is  an isomorphism $f:(A,\iota)\to (A',\iota')$ of the underlying $\OL$-abelian schemes over $S$ such that $\varphi(\lambda)=f^\vee\varphi'(\lambda)f$ for all $\lambda\in I$. 

We say that a stack is a Hilbert moduli stack if it is equal to the stack $\mathcal M^I$ associated to some $L$, $I$ and $I_+$ as above. An important example of a Hilbert moduli stack is given by the moduli stack $\mathcal M_{1,1}$ of elliptic curves, which was introduced and studied by Mumford~\cite[$\mathsection 4$]{mumford:picardmoduli} and Deligne--Rapoport~\cite{dera:modules}. If $I=D^{-1}$ is the $\ZZ$-dual of $\OL$ and $I_+$ is the (standard) positivity notion $(+,\dotsc,+)$ on $I$, then we call $\mathcal M^I$ the Hilbert moduli stack associated to $L$.  For example $\mathcal M_{1,1}$ identifies with the Hilbert moduli stack associated to $\QQ$. To simplify notation, we shall often omit $I$ from the notation and we  write $\mathcal M$ for $\mathcal M^I$.

\paragraph{Properties of $I$-polarizations.} To discuss some aspects of (iii), let $(A,\iota)$ be an $\OL$-abelian scheme of relative dimension $g$ over an arbitrary scheme $S$. A morphism $\varphi$ as in (iii) is called an $I$-polarization of $(A,\iota)$.  In the case when $\Delta=|\textnormal{disc}(L/\QQ)|$  is invertible on $S$, the existence of such an $I$-polarization assures (see \cite[Cor 2.9]{depa:hilbertmodular}) that the Lie algebra $\textnormal{Lie}(A)$ of $A$ is, locally on $S$, a free $\OL\otimes_\ZZ\mathcal O_S$-module of rank one.  In particular, if $\Delta$ is invertible on $S$ then $\mathcal M(S)$ identifies with $\mathcal M'(S)$ where we denote by $\mathcal M'$ the stack defined by Rapoport in~\cite[1.19]{rapoport:hilbertmodular}. 
  Further,  Yu~\cite{yu:hilbmod} and Vollaard~\cite{vollaard:hilbmod}  worked out in detail the relation of the existence of an $I$-polarization of $(A,\iota)$ to various conditions in the literature \cite{rapoport:hilbertmodular,kottwitz:shimvar,depa:hilbertmodular}, including the determinant condition. 


\paragraph{Geometry of $\mathcal M$.} We now record some useful properties of $\mathcal M$. Firstly $\mathcal M$ is a DM-stack, that is $\mathcal M$ is an algebraic stack in the sense of Deligne--Mumford~\cite[Def 4.6]{demu:stacks}. 
Furthermore $\mathcal M$ is separated and of finite type over $\ZZ$. 
 The stacks $\mathcal M$ and $\mathcal M'$ coincide over $\ZZ[1/\Delta]$, and therefore Rapoport~\cite[Thm 1.20]{rapoport:hilbertmodular} gives that $\mathcal M$ is smooth over $\ZZ[1/\Delta]$.   Deligne--Pappas~\cite{depa:hilbertmodular} studied the singularities of $\mathcal M$ over any rational prime $p$ dividing $\Delta$. In particular they showed that $\mathcal M$ is flat and relative local complete intersection over $\ZZ$, and that the fiber of $\mathcal M$ over $p$ is smooth outside a closed subset of codimension two. Indeed this follows from \cite[Thm 2.2]{depa:hilbertmodular} which says that for any integer $n\geq 3$ these statements hold over $\ZZ[1/n]$ for the finite \'etale cover $\mathcal M_n$ of $\mathcal M_{\ZZ[1/n]}$, where $\mathcal M_n$ is the (representable) Hilbert moduli stack over $\ZZ[1/n]$ with principal level $n$-structure.

\subsection{Moduli schemes}\label{sec:hms2}
Let $\mathcal M$ be a Hilbert moduli stack. We now define Hilbert moduli schemes. Let $Y$ be a scheme and let $\mathcal P$ be a presheaf on $\mathcal M$. We say that $Y$ is a Hilbert moduli scheme of $\mathcal P$, and we write $Y=M_{\mathcal P}$, if $\mathcal P$ is representable by an object in $\mathcal M(Y)$.  Sometimes we simply say that $Y=M_{\mathcal P}$ is a Hilbert moduli scheme when $Y$ is a Hilbert moduli scheme of a presheaf $\mathcal P$ on $\mathcal M$. When working with Hilbert moduli schemes it is often important to specify the involved presheaf,  since $Y$ can be a Hilbert moduli scheme of presheaves on $\mathcal M$ which are geometrically very different.   For example if $\mathcal M(Y)$ is nonempty, then  $Y$ is a Hilbert moduli scheme of the presheaf $\mathcal P=h_y$ on $\mathcal M$ for each object $y$ of $\mathcal M(Y)$. In the case when  $\mathcal M=\mathcal M_{1,1}$, the Hilbert moduli schemes are precisely the moduli schemes of elliptic curves defined in  Katz--Mazur~\cite[4.3.1]{kama:moduli} and \cite[$\mathsection$3.1]{rvk:intpointsmodellhms}.   

\paragraph{Moduli interpretation.} We call a presheaf $\mathcal P$ on $\mathcal M$ a moduli problem on $\mathcal M$. Further, for each object $x$ of $\mathcal M$, the elements of the set $\mathcal P(x)$ are called $\mathcal P$-level structures on $x$. The terminology moduli scheme is (partly) motivated by the following formal  lemma which gives a moduli intepretation for the points of any Hilbert moduli scheme. 

\begin{lemma}\label{lem:moduli}Suppose that $Y$ is a Hilbert moduli scheme of a presheaf $\mathcal P$ on $\mathcal M$. Then $Y$ represents the presheaf on $\sch$ which sends any scheme $S$ to the set of isomorphism classes of pairs $(x,\alpha)$ with $x$ an object of $\mathcal M(S)$ and $\alpha\in\mathcal P(x)$.
\end{lemma}

Here pairs $(x,\alpha)$ and $(x',\alpha')$ are isomorphic if there exists an isomorphism $f:x\to x'$ in $\mathcal M(S)$ which satisfies $\alpha=\mathcal P(f)(\alpha')$. Further, formal computations show that sending $S$ to the set of isomorphism classes of pairs $(x,\alpha)$  as in Lemma~\ref{lem:moduli}  indeed defines a presheaf  on $\sch$. These computations only use that $\mathcal M$ is fibered in groupoids over $\sch$ in the sense of \cite[$\mathsection$4]{demu:stacks}. 
In fact the following proof shows that Lemma~\ref{lem:moduli} holds more generally when $\mathcal M$ is replaced by an arbitrary category fibered in groupoids over $\sch$. 

\begin{proof}[Proof of Lemma~\ref{lem:moduli}]
Let $p:\mathcal X\to \sch$ be a category fibered in groupoids and let $\mathcal P$ be a presheaf on $\mathcal X$ which is representable by an object $y$ in $\mathcal X(Y)$. We fix a choice of pullbacks in $\mathcal X$. Let $S$ be a scheme, let $\varphi$ be in $Y(S)$ and let $\alpha:\varphi^*y\to y$ be the pullback of $y$ along $\varphi:S\to Y$.  We identify $\mathcal P$ with the presheaf $h_y= \Hom_{\mathcal X}(-,y)$ on $\mathcal X$. It follows that $\varphi\mapsto  [(\varphi^*y,\alpha)]$
defines a map from $Y(S)$ to the set of isomorphism classes of pairs $(x,\alpha)$ with $x$ an object of $\mathcal X(S)$ and $\alpha$ an element of $\mathcal P(x)=h_y(x)$. 
This map is injective: If $\psi$ lies in $Y(S)$ such that $(\varphi^*y,\alpha)$ is isomorphic to $(\psi^*y,\alpha')$,  then there exists an isomorphism $f:\varphi^*y\isomto \psi^*y$ in $\mathcal X(S)$ with $\alpha=\alpha'f$ and thus $\varphi=p(\alpha)$ equals $p(\alpha'f)=\varphi'$ as desired.  
 Furthermore, the two axioms of a category fibered in groupoids assure that $\varphi\mapsto  [(\varphi^*y,\alpha)]$ is surjective (take $\varphi=p(\alpha)$) and is canonical in $S$ by compatibility of pullbacks. We conclude that $Y$ represents the presheaf as claimed in the lemma.
\end{proof}

We now consider an arbitrary moduli problem $\mathcal P$ on $\mathcal M$ and we let $S$ be a scheme. To control the number of distinct $\mathcal P$-level structures on any object in $\mathcal M(S)$, we define
\begin{equation}\label{def:maxlvl}
|\mathcal P|_S=\sup |\mathcal P(x)|
\end{equation}
with the supremum taken over all objects $x$ of $\mathcal M(S)$. 
We say that $\mathcal P$ is finite over $\mathcal M(S)$ if  $\mathcal P|_{\mathcal M(S)}$ is a finite presheaf. In particular  $\mathcal P$ is finite over $\mathcal M(S)$ if $|\mathcal P|_S<\infty$.

\paragraph{Forgetful map.} Suppose that $Y=M_{\mathcal P}$ is a Hilbert moduli scheme. To define its natural forgetful map, we identify the scheme $Y$ with its functor of points. We  denote by $\{(A,\iota,\varphi,\alpha)\}/\sim$ the presheaf on $\sch$ which sends any scheme $S$ to the set of isomorphism classes of quadruples $(A,\iota,\varphi,\alpha)$ over $S$. Here a quadruple over $S$ is a pair $(x,\alpha)$ given by an object $x=(A,\iota,\varphi)$  of $\mathcal M(S)$ together with  a $\mathcal P$-level structure $\alpha$ on $x$. In other words $\{(A,\iota,\varphi,\alpha)\}/\sim$ is the presheaf  in Lemma~\ref{lem:moduli} and thus Lemma~\ref{lem:moduli} gives
\begin{equation}\label{quadriso}
Y\isomto\{(A,\iota,\varphi,\alpha)\}/\sim.
\end{equation}
Let $g\geq 1$ be an integer. We denote by $\absg$ the presheaf on $\sch$ which sends a scheme $S$ to the the set of isomorphism classes of abelian schemes over $S$ of relative dimension $g$. This underlined $\absg$ should not be confused with the usual $A_g=A_{g,1}$ which classifies principally polarized abelian schemes $(A,\psi)$, see  the discussion of the forgetful maps $A_{g,d}\to \absg$ in \eqref{eq:forgetfulpolav}. Suppose now that $g$ is the relative dimension of the abelian schemes classified by the stack $\mathcal M$. Then, on projecting the isomorphism class of a quadruple to the isomorphism class of the underlying abelian scheme, we see that \eqref{quadriso} induces a morphism 
\begin{equation}\label{def:forgetfulmap}
\phi:Y\to \absg
\end{equation}
of presheaves on $\sch$. Intuitively one can think of $\phi$ as the map which forgets the extra structures $(\iota,\varphi,\alpha)$ on the abelian scheme $A$ underlying the quadruple $(A,\iota,\varphi,\alpha)$ defining a point of $Y$.  In light of this we call $\phi$ the (natural) forgetful map of $Y=M_{\mathcal P}$, and we write $\phi_\mathcal P$ for $\phi$ in situations where it is important to specify $\mathcal P$.


\paragraph{Height.} We now generalize to arbitrary Hilbert moduli schemes the height which was defined in \cite[(3.3)]{rvk:intpointsmodellhms} for moduli schemes of elliptic curves. Let $S$ be a connected Dedekind scheme whose function field $k$ is algebraic over $\QQ$. For any abelian scheme $A$ over $S$ we denote by 
$
h_F(A)
$
the stable Faltings height $h_F$  of the generic fiber of $A$, defined for example in \cite[$\mathsection$2.1]{rvk:modularhms}. Here we use Faltings' original normalization~\cite[p.354]{faltings:finiteness} of the metric involved in the definition of $h_F$. In fact we obtain a height function $h_F:\absg(S)\to \mathbb R$ since isomorphic abelian schemes over $S$ have the same stable Faltings height. Then for any Hilbert moduli scheme $Y=M_{\mathcal P}$ with forgetful map $\phi=\phi_{\mathcal P}$, we  define a height function
\begin{equation}\label{def:height}
h_\phi: Y(S)\to \mathbb R
\end{equation}
by setting $h_\phi=\phi^*h_F$. In other words $h_\phi$ is the pull back of the stable Faltings height $h_F:\absg(S)\to \mathbb R$ by the map $\phi(S):Y(S)\to \absg(S)$. The set $Y(S)$ might be  empty  and to cover this case we define $h_\phi=0$ on the empty set. The height $h_\phi$ in \eqref{def:height} is compatible with any base change $S'\to S$ for $S'$ a connected Dedekind scheme whose function field is algebraic over $\QQ$, since $\phi$ is a morphism of presheaves and since $h_F$ is invariant under $S'\to S$. Further, if $S$ and $T$ are nonempty open subschemes of $\spec(\ZZ)$, and if $Y$ is a scheme such that $Y_T=M_{\mathcal P}$ is a Hilbert moduli scheme, then we define the height $h_\phi$ on $Y(S)$ as follows:  For any $P\in Y(S)$ we put $h_{\phi}(P)=h_{\phi}(P'),$ where $P'$ is the canonical image of $P$ in $Y(U)\cong Y_T(U)$ for $U=S\cap T$ and where  $h_\phi:Y_T(U)\to \RR$ is the above height \eqref{def:height} involving the forgetful map $\phi=\phi_{\mathcal P}$ of $Y_T$.  We point out that the height $h_\phi$ depends on $\phi$ which in turn depends on the involved presheaf $\mathcal P$. 

\paragraph{Hilbert moduli schemes over $\ZZ[1/n]$.}Sometimes we shall work over an open substack $\mathcal M_T$ of $\mathcal M$ with $T\subseteq \spec(\ZZ)$ nonempty open: On replacing in the above definitions and constructions $\mathcal M$ by $\mathcal M_T$, we directly obtain the notion of a Hilbert moduli scheme over $T$ of a presheaf on $\mathcal M_T$ and then over $T$ we may and do define the forgetful map $\phi$ and the height $h_\phi$ analogously as in \eqref{def:forgetfulmap} and \eqref{def:height} respectively. We say that a presheaf $\mathcal P$ on $\mathcal M$ is representable over $T$ if its restriction to $\mathcal M_T$ is representable.


\paragraph{Northcott property.} Let $K$ be a number field with ring of integers $\OL_K$. Suppose that $Y$ is a variety over $\OL_K$ and assume that $Y$ is a Hilbert moduli scheme of some presheaf $\mathcal P$ on $\mathcal M$. Then we say that the height $h_\phi$ defined in \eqref{def:height} has the Northcott property on $Y(\bar{K})$ if the following holds\footnote{Here we view $Y$ as an $\OL_K$-scheme. Then $Y(\bar{K}):=\Hom_{\OL_K}(\bar{K},Y)$ by the usual convention and the height $h_\phi:\Hom_\sch(\bar{K},Y)\to\RR$ constructed in \eqref{def:height} is well-defined on the subset $Y(\bar{K})\subseteq \Hom_\sch(\bar{K},Y)$.}: For any real $c>0$ and each $d\in\ZZ_{\geq 1}$, there are only finitely many points $P$ in $Y(\bar{K})\cong Y_K(\bar{K})$ with $[K(P):K]\leq d$ and $h_\phi(P)\leq c$. 



\begin{proposition}\label{prop:northcott}
If $\mathcal P$ is  finite over $\mathcal M(\bar{K})$, then $h_\phi$ has the Northcott property on $Y(\bar{K})$.
\end{proposition}

In particular $h_\phi$ has the Northcott property on $Y(\bar{K})$ if $|\mathcal P|_{\bar{K}}<\infty$. We shall deduce Proposition~\ref{prop:northcott} from the Northcott property \cite[p.169]{fach:deg} of $h_F$ on $\absg(\bar{\QQ})$ by exploiting the following: The  map $\phi(S)$ used in \eqref{def:height} has finite fibers if and only if $\mathcal P$ is finite over $\mathcal M(S)$. This statement, which uses inter alia the Jordan--Zassenhaus theorem, will be obtained in Lemma~\ref{lem:finiteiff} as a byproduct of the proofs of our main results.


\section{Statement of main result}\label{sec:results}
To state our main result we use the terminology introduced above. Let $L$ be a totally real number field  with ring of integers $\OL$, let $I$ be a nonzero finitely generated $\OL$-submodule of $L$ with a positivity notion $I_+$ and let $\mathcal M=\mathcal M^I$ be the associated Hilbert moduli stack. Write $g=[L:\QQ]$ and for any $U\subseteq \spec(\ZZ)$ denote by $N_U=\prod p$ the product of all rational primes $p$ not in $U$. Let $Y$ be a variety over $\ZZ$ and let $S\subseteq\spec(\ZZ)$ be nonempty open. 

\begin{theorem}\label{thm:main}
 Suppose that there is a nonempty open $T\subseteq \spec(\ZZ)$ such that $Y_T$ is a Hilbert moduli scheme of a presheaf $\mathcal P$ on $\mathcal M$. If $U=S\cap T$ then the following holds.
\begin{itemize}
\item[(i)]  Any $P\in Y(S)$ satisfies $h_\phi(P)\leq (3g)^{144g}N_U^{24}.$
\item[(ii)] If $e_g=(8g)^{8}$ then  
$|Y(S)|\leq |\textnormal{Pic}(\OL)||\mathcal P|_{U}(2N_U)^{e_g}.$
\end{itemize}
\end{theorem}
It holds that $N_U=\textnormal{rad}(N_T N_S)$, and  $|\mathcal P|_{U}\leq |\mathcal P|_\QQ\leq |\mathcal P|_{\bar{\QQ}} \leq |\mathcal P|_\CC$ by Lemma~\ref{lem:geomlvlred}. Proposition~\ref{prop:northcott} gives that the height $h_\phi$ on $Y(S)$, which depends on $\mathcal P$,  has the Northcott property on $Y(\bar{\QQ})$ if $|\mathcal P|_{\bar{\QQ}}<\infty$. However the bound for $h_\phi$ in (i) can also be useful when $h_\phi$ has no Northcott property; see for example the introduction where we discuss various aspects of Theorem~\ref{thm:main}. We obtain Theorem A in the introduction by applying Theorem~\ref{thm:main} with $I=D^{-1}$ the inverse different of $\OL$ and $I_+$ the standard positivity notion. To see this take $T=\spec(\ZZ[1/\nu])$ for $\nu=N_T$ and consider the complement of $S\subseteq \spec(\ZZ)$ which is a finite set of rational primes.  
The proof of Theorem~\ref{thm:main}  will be given in Section~\ref{sec:proofs}, and we refer to the introduction for an outline of the principal ideas used in the proof.



\subsection{The Hilbert moduli scheme $Y_1(\mathfrak n)$}\label{sec:p1n}

We continue our notation. To discuss a first application of Theorem~\ref{thm:main}, we take an ideal $\mathfrak n\subseteq \OL$ of norm $n\geq 1$ and we consider the  moduli problem $\mathcal P_1(\mathfrak n)$ on $\mathcal M$. Intuitively this moduli problem parametrizes $\mathfrak n$-torsion points of exact order $\mathfrak n$. For example, if $\mathcal M=\mathcal M_{1,1}$ then  $\mathcal P_1(\mathfrak n)$ identifies with the $\Gamma_1(n)$-moduli problem of Katz--Mazur~\cite[p.99]{kama:moduli} which sends an elliptic curve $E\in\mathcal M_{1,1}(\CC)$ to its set of points of exact order $n$. 

\paragraph{The moduli problem $\mathcal P_1(\mathfrak n)$.}  We now give the formal definition of the presheaf $\mathcal P_1(\mathfrak n)$ on $\mathcal M$ following Drinfeld, Katz--Mazur and Pappas~\cite[$\mathsection$2.3.2]{pappas:hilbmod}. Let $S$ be a scheme,  let $x=(A,\iota,\varphi)$ be in $\mathcal M(S)$ and let $y=(A',\iota',\varphi')$ be in $\mathcal M^{I'}(S)$ for $I'=\mathfrak nI$. An $\mathfrak n$-isogeny $f:x\to y$ is a morphism of $\OL$-abelian schemes $f:(A,\iota)\to (A',\iota')$ which is an isogeny of (constant) degree $n$ such that $\ker(f)$ is annihilated by $\mathfrak n$ and such that $f^\vee\varphi'(\lambda)f=\varphi(\lambda)$ for each $\lambda\in I'\subseteq I$. Further, an $\OL/\mathfrak n$-generator of $G=\ker(f)$ is a section $P\in G(S)$ such that the $\OL/\mathfrak n$-linear morphism $\epsilon:(\OL/\mathfrak n)_S\to G$, induced by $1\mapsto P$, is an $\OL/\mathfrak n$-structure on $G$ in the sense of \cite[$\mathsection$1.10]{kama:moduli} or \cite[5.1.1]{pappas:hilbmod};  in the case when $G$ is \'etale over $S$, the morphism $\epsilon$ is an $\OL/\mathfrak n$-structure on $G$ if and only if $\epsilon$ is an isomorphism. We define
\begin{equation}\label{def:p1n}
\mathcal P_1(\mathfrak n)(x)=\{P\in A(S); \, P \textnormal{ an }  \OL/\mathfrak n\textnormal{-generator of }\ker(f)\}
\end{equation}
where $f:x\to y$ is an $\mathfrak n$-isogeny with $y\in \mathcal M^{I'}(S)$. Being an $\mathfrak n$-isogeny and being an $\OL/\mathfrak n$-generator are both properties which are stable under arbitrary base change. Therefore we see that sending the object $x$ of $\mathcal M$ to the set $\mathcal P_1(\mathfrak n)(x)$ defines a presheaf $\mathcal P_1(\mathfrak n)$ on $\mathcal M$. An ideal of $\OL$ is called square-free if it is the product of distinct prime ideals of $\OL$. In the case when the ideal $\mathfrak n$ is square-free, Pappas~\cite{pappas:hilbmod} studied the singularities of the separated finite type Deligne--Mumford stack $\mathcal H_{00}(\mathfrak n)$ over $\ZZ$ whose objects over $S$ are $\mathfrak n$-isogenies  $f:x\to y$ together with an $\OL/\mathfrak n$-generator $P$ of $\ker(f)$. For example, Pappas showed that $\mathcal H_{00}(\mathfrak n)$ is Cohen--Macaulay and flat over $\ZZ$ if the ideal $\mathfrak n$ is square-free, and  he proved that $\mathcal H_{00}(\mathfrak n)$ is regular if $g=2$ and the norm $n$ is square-free. 

\paragraph{Integral points.}To apply Theorem~\ref{thm:main}, we need that the moduli problem  $\mathcal P_1(\mathfrak n)$ on $\mathcal M$ is representable over some nonempty open of $\spec(\ZZ)$. In \cite{vkkr:repcond} we will conduct some effort to work out in detail explicit conditions on $\mathfrak n$ which assure the representability of $\mathcal P_1(\mathfrak n)$. Suppose now that $\mathcal P_1(\mathfrak n)$ is representable over $\ZZ[1/n]$ with Hilbert moduli scheme $Y_1(\mathfrak n)=M_{\mathcal P_1(\mathfrak n)}$ a (quasi-projective) variety over $\ZZ[1/n]$. Then we shall deduce in Section~\ref{sec:proofp1n} the following result in which we assume again that $S\subseteq \spec(\ZZ)$ is nonempty open.

\begin{corollary}\label{cor:p1n}
Let $Y$ be a variety over $\ZZ$, which becomes over $\ZZ[1/n]$ isomorphic to $Y_1(\mathfrak n)$. 
\begin{itemize}
\item[(i)]  Any point $P\in Y(S)$ satisfies $h_\phi(P)\leq (3g)^{144g}(nN_S)^{24}.$
\item[(ii)] The cardinality of  
$Y(S)$ is at most $|\textnormal{Pic}(\OL)|(2nN_S)^{e_g}.$
\end{itemize}
\end{corollary}

In the case when $\mathcal M$ is the Hilbert moduli stack associated to $L$, this result is precisely the corollary discussed in the introduction.  If $\mathcal M=\mathcal M_{1,1}$ and $n\geq 4$, then one can compute explicit equations of the curve $Y_1(n)$ over $\ZZ[1/n]$; see for example Baaziz~\cite{baaziz:x1nmodels}, Sutherland~\cite{sutherland:torsion} and Jin~\cite{jin:torsion}. We are currently trying to find some `interesting' Diophantine equations which define a Hilbert moduli scheme $Y_1(\mathfrak n)$ over $\ZZ[1/n]$ when $g\geq 2$.

\section{Endomorphisms of abelian schemes}\label{sec:endopreliminaries}

In this section we collect preliminary results for  endomorphisms of abelian schemes over Dedekind schemes. We first give explicit relations between the endomorphism rings of two isogenous abelian schemes. Then we review a variant of the (Serre) tensor product of abelian schemes with certain not necessarily projective modules. In the last part we consider in more detail the endomorphism ring of an abelian scheme of $\GL_2$-type.

\subsection{Endomorphism rings of isogenous abelian schemes}\label{sec:endorelation}

Let $S$ be a connected Dedekind scheme. While the endomorphism algebras of isogenous abelian schemes $A$ and $A'$ over $S$ coincide up to isomorphism, the relation between the endomorphism rings encapsulates interesting arithmetic information.  In this section, we review explicit constructions of isomorphisms $\End^0(A)\cong\End^0(A')$ in order to obtain basic results (used in Section~\ref{sec:endo}) relating $\End(A)$ with $\End(A')$.  These constructions and results should be all well-known. However, in several cases we are not aware of suitable references and then we will give the arguments for the sake of completeness.


\paragraph{Quotients and isogenies.} Forming quotients of group schemes is a delicate process in general. However, in what follows we only consider quotients in the following special situation. Let $G\subseteq A$ be a closed subgroup scheme which is finite flat over $S$. 
It follows for example from results in \cite[$\mathsection$4]{anantharaman:groupes} that the fppf quotient of $A$ by $G$ is represented by a commutative $S$-group scheme $A/G$.
The quotient morphism $A\to A/G$ is finite and faithfully flat, since  $G$ is finite flat over $S$. 
Then, on using that  the structure morphism $A\to S$ is smooth proper, we obtain that $A/G$ is smooth proper over $S$.
In particular $A/G$ is an abelian scheme over $S$ and the quotient morphism $A\to A/G$ is an isogeny.

Let $\varphi \colon A\to A'$ be a morphism of abelian schemes over $S$. The following three statements are equivalent: (i)  $\varphi$ is finite surjective; (ii)  $\varphi_s:A_s\to A'_s$ is finite surjective for each $s\in S$; (iii) $\varphi_k:A_k\to A'_k$ is finite surjective for $k$ the function field of $S$. In particular, $\varphi$ is an isogeny  if and only if $\varphi_k$ is an isogeny. For the readers convenience we now explain why (iii) implies (i). For each nonzero $n\in\ZZ$ the multiplication by $n$ morphism $[n]:A\to A$ is finite flat by \cite[$\mathsection$7.3]{bolura:neronmodels}. If $d=\deg(\varphi_k)$ then (iii) assures that $[d]_k$ factors through $\varphi_k$ and thus \eqref{eq:neron} shows that the finite morphism $[d]$ has to factor through the unique extension $\varphi$ of $\varphi_k$. This implies that $\varphi$ is finite, since any morphism between proper $S$-schemes is proper. In particular each $\varphi_s$ is finite and hence surjective, which proves (i).  Further,  (i) implies that $\varphi$ is flat. Therefore, if $\varphi$ is an isogeny then its kernel $\ker(\varphi)$ is a closed subgroup scheme of $A$ which is finite flat over $S$, there is an isomorphism $A/\ker(\varphi)\isomto A'$ and the rank $\deg(\varphi)$ of $\ker(\varphi)$ is constant on the connected scheme $S$.

\paragraph{Endomorphism rings.} Let $\varphi \colon A\to A'$ be an isogeny of abelian schemes over $S$ and write $d=\deg(\varphi)$. For each nonzero $n\in\ZZ$, we denote by $A_n$ the kernel of the isogeny $[n]:A\to A$. 
The above discussion shows that $\ker(\varphi)\subseteq A_d$ and there is a natural projection $A/\ker(\varphi)\to A/A_d$. 
We consider the isogeny $\varphi':A'\to A$ given by 
\begin{equation}\label{eq:naturalconverse}
\varphi': A' \isomto A/\ker(\varphi) \to A/A_d  \isomto A.
\end{equation} 
Here the first morphism is the inverse of the natural quotient map $A/\ker(\varphi)\isomto A'$ 
and the third morphism is $[d]$. If $S$ is the spectrum of a field, then \cite[Lem~7.3.5]{bolura:neronmodels} implies that $\varphi'\varphi=[d]$. Moreover, on using the above arguments and elementary formal computations, we see that the isogeny $\varphi'$ is an inverse of $\varphi$ in the following sense.

\begin{lemma}\label{lem:dmult}
It holds
$\varphi'\varphi = [d]$ inside $\End(A)$ and $\varphi\varphi' = [d]$ inside $\End(A')$.
\end{lemma}


\noindent To define  an explicit isomorphism $\End^0(A) \isomto \End^0(A')$, we introduce more notation.  Let $B$, $B'$ and $B''$ be abelian schemes over $S$. Composition induces a pairing $\circ$ from $\Hom^0(B,B')\times\Hom^0(B',B'')$ to $\Hom^0(B,B'')$. 
For any nonzero $m\in\ZZ$ and each $\psi\in\Hom(B,B')$, we denote by $\psi/m$ the element $\psi\otimes m^{-1}$ of $\Hom^0(B,B').$ It follows for example from \eqref{eq:neron} and \cite[p.176]{mumford:av} that $\Hom(B,B')$ is a finitely generated free abelian group. Thus $\psi\mapsto \psi\otimes 1$ embeds  $\Hom(B,B')$  into $\Hom^0(B,B')$. 
In what follows we shall often identify a morphism $B\to B'$ with its image in $\Hom^0(B,B')$. We now define 
\begin{equation}\label{def:phiupstar}
\varphi^* \colon \End^0(A) \isomto \End^0(A'), \quad f \mapsto \varphi \circ f \circ \varphi'/d. 
\end{equation}
To prove that $\varphi^*$ is an isomorphism of $\QQ$-algebras, we use that the endomorphism $[d]$ of $A$ is invertible inside the $\QQ$-algebra $\End^0(A)$ with inverse $1 \otimes d^{-1}$. Then Lemma~\ref{lem:dmult} implies that $\varphi'/d\circ \varphi$ and $\varphi\circ \varphi'/d$ are the multiplicative units of $\End^0(A)$ and $\End^0(A')$ respectively. This shows that $\varphi^*$ is an isomorphism of $\QQ$-algebras, 
with inverse
\begin{equation}\label{def:philowstar}
\varphi_* \colon \End^0(A') \isomto \End^0(A), \quad g \mapsto \varphi'/d \circ g \circ \varphi.
\end{equation}
On using these explicit constructions of the isomorphisms $\varphi_*$ and $\varphi^*$ between the endomorphism algebras of $A$ and $A'$, we now can relate the endomorphism rings as follows. 

\begin{lemma}\label{lem:endorelation}
Suppose that $\varphi:A\to A'$ is an isogeny of abelian schemes over $S$ of degree $d=\deg(\varphi)$.  Then the following two statements hold.
\begin{itemize}
\item[(i)] Inside $\End^0(A')$ we have $d\cdot \End(A') \subseteq \varphi^*(\End(A))$.
\item[(ii)] Inside $\End^0(A)$ it holds $d\cdot \End(A) \subseteq \varphi_*(\End(A'))$. 
\end{itemize}
\end{lemma}
\begin{proof}
To prove (i) we take $g \in \End(A')$. Then $f = \varphi' g \varphi$ lies in $\End(A)$ and we compute that  $f= d\varphi_*(g)$ inside $\End^0(A)$. 
It follows that $\varphi^*(f) = d\varphi^*\varphi_*(g) = dg$ inside $\End^0(A')$, since $\varphi^*$ is an isomorphism of $\QQ$-algebras with inverse $\varphi_*$. In particular, for each  $g \in \End(A')$ there exists an $f \in \End(A)$ such that $\varphi^*(f) = d g$ inside $\End^0(A')$. We deduce that $d \cdot \End(A') \subseteq \varphi^*(\End(A))$ inside $\End^0(A')$. This proves assertion (i). 

To show (ii) we take $f\in \End(A)$. Then $g = \varphi f \varphi'$ lies in $\End(A')$ and we see that $\varphi_*(g)=\varphi_*\varphi^*(fd)=fd=df$ inside $\End^0(A)$. Hence $d \cdot \End(A) \subseteq \varphi_*(\End(A'))$ inside $\End^0(A)$ 
as claimed in assertion (ii). This completes the proof.
\end{proof}

\subsection{Tensor product of abelian schemes}\label{sec:tensorprod}

Let $A$ be a nonzero abelian scheme over an arbitrary connected Dedekind scheme $S$ and let $\mathcal O$ be a commutative subring of $\End(A)$. 
Suppose that $\mathcal O'$ is a $\ZZ$-flat $\OL$-algebra that is finitely generated as an $\OL$-module and $\OL\otimes_\ZZ\QQ\isomto \OL'\otimes_\ZZ\QQ$ is an isomorphism. The following lemma is (a consequence of a construction) due to Chai--Conrad--Oort~\cite{chcooo:cmbook}.


\begin{lemma}\label{lem:abtensorproduct}
There is an abelian scheme $A'$ over $S$, isogenous to $A$, with $\mathcal O'\hookrightarrow \End(A')$.
\end{lemma}

Here we can take $A'=\OL'\otimes_\OL A$ if the $\OL$-module $\OL'$ is projective. However, since  $\mathcal O'$  is not necessarily projective, we can in general not apply the usual (Serre) tensor product construction for projective $\mathcal O$-modules. To prove Lemma~\ref{lem:abtensorproduct}, we use instead that $A$ is the N\'eron model of its generic fiber and we work with the fppf sheaf $A_k'=\mathcal O'\otimes_{\mathcal O} A_k$ which has the desired properties for any abelian variety $A_k$ over a field $k$ by \cite[$\mathsection$1.7.4]{chcooo:cmbook}.

\begin{proof}[Proof of Lemma~\ref{lem:abtensorproduct}]
Let $k$ be the function field of our connected Dedekind scheme $S$ and let $A_k$ be the generic fiber of $A$. On using the isomorphism $\End(A)\isomto \End(A_k)$ in \eqref{eq:neron} and the subring $\mathcal O$ of $\End(A)$, we equip $A_k$ with an $\OL$-abelian scheme structure. Now we see that   $T\mapsto \mathcal O'\otimes_{\mathcal O} A_k(T)$ defines a contravariant functor  from the category of $k$-schemes to the category of abelian groups.   Then  \cite[1.7.4.5]{chcooo:cmbook} gives that the fppf sheafification of this functor  is represented by an abelian variety $A'_k=\mathcal O'\otimes_{\mathcal O}A_k$ over $k$ which is isogenous to $A_k$ and $\mathcal O'\hookrightarrow \End(A'_k)$. 
 The generic fiber $A_k$ of the abelian scheme $A$ over $S$ has good reduction at all closed points of $S$. Hence the isogenous abelian variety $A'_k$ has good reduction at all closed points of $S$ and thus it extends to an abelian scheme $A'=\mathcal O'\otimes_{\mathcal O}A$ over $S$ with $\End(A')\cong \End(A'_k)$.  Furthermore any isogeny between $A_k$ and $A'_k$ extends to an isogeny between $A$ and $A'$. Thus $A'$ has the desired properties.
\end{proof}
For non-projective modules $\OL'$ and a general base scheme $S$, the construction of a `tensor product' of  $A$ with $\OL'$ is a complicated problem; see  \cite[1.7.4.3]{chcooo:cmbook}.


\subsection{Endomorphism algebras of abelian schemes of $\GL_2$-type}
In this section we collect useful results on  abelian schemes of $\GL_2$-type. In particular we review and prove some properties of the endomorphism algebras of such abelian schemes. Throughout this subsection we let $S$ be a scheme, we denote by $g$ a positive rational integer and we let $A$ be an abelian scheme over $S$ of relative dimension $g$.

\paragraph{Abelian schemes of $\GL_2$-type.} We say\footnote{Some authors use a more restrictive definition. For example, they assume in addition simplicity or they make extra assumptions on the Lie algebra. Assuming simplicity would be too restrictive for our purpose. On the other hand, we do not make extra assumptions on the Lie algebra because either we do not need them or these assumptions are automatically satisfied in the situations under consideration.} that $A$ is of $\GL_2$-type if the $\QQ$-algebra  $\End^0(A)$ contains a number field of degree $g$ over $\QQ$.  The terminology is motivated as follows: If $\End^0(A)$ contains a number field $L$ of degree $g$ over $\QQ$, then for each $s\in S$ the field $L$ embeds into $\End^0(A_s)$ 
 and thus the action of the absolute Galois group of $k(s)$ on the rational Tate module $V_\ell(A_s)$ defines a representation with values in $\GL_2(L\otimes_\QQ \QQ_\ell)$ where $\ell\neq \textnormal{char} \, k(s)$ is a rational prime. 
More generally, we say that $A$ is of product $\GL_2$-type if $A$ is isogenous to a product $\prod A_i$ of abelian schemes $A_i$ over $S$ such that each $A_i$ is of $\GL_2$-type.  
We denote by $M_{\GL_2,g}(S)$ the set of isomorphism classes of abelian schemes over $S$ of relative dimension $g$ which are of product $\GL_2$-type.
Then we claim that 
\begin{equation}\label{def:gl2functor}
S\mapsto M_{\GL_2,g}(S)
\end{equation}
defines a presheaf $M_{\GL_2,g}$ on $\sch$, 
where for any morphism $S'\to S$ of schemes the map $M_{\GL_2,g}(S'\to S)$ sends the isomorphism class of $A$ in $M_{\GL_2,g}(S)$ to the isomorphism class of $A_{S'}$ in $M_{\GL_2,g}(S')$.  If $\End^0(A)$ contains a subring $L$, then any base change $S'\to S$ induces a morphism of rings $L\to\End^0(A_{S'})$ 
 which is injective when $L$ is a field. We conclude that being of $\GL_2$-type is a property which is stable under any base change.  Then the same holds for the property of being of product $\GL_2$-type, since finite and faithfully flat (and thus being an isogeny) are properties of morphisms which are stable under any base change. This shows that \eqref{def:gl2functor} indeed defines a presheaf on $\sch$.


\paragraph{Endomorphisms.} We start with an elementary observation which we shall use several times in what follows in this paper. Let $n\geq 1$ be a rational integer. For any ring $R$ we denote by $\uM_n(R)$ the endomorphism ring of the free $R$-module $R^n$. 

\begin{lemma}\label{lem:Centralizer}
Let $K$ and $L$ be number fields such that $[L:\QQ]=n[K:\QQ]$. If $L$ is a subring of  $\uM_n(K)$, then $L$ is the centralizer of $L$ in $\uM_n(K)$ and thus $K\subseteq L$. 
\end{lemma}
\begin{proof}
The abelian group $V=K^n$ becomes a $L$-vector space via  the inclusion $L \subseteq \uM_n(K)$. Linear algebra gives that $[K:\QQ]\dim_K V=\dim_\QQ V=[L:\QQ]\dim_L V$, and our assumption $[L:\QQ]=n[K:\QQ]$ then implies that $\dim_L V=1$. 
Each element in the centralizer $C$ of $L$ in $\uM_n(K)$ is an endomorphism of the $L$-vector space $V$, 
and the endomorphism ring of any one-dimensional $L$-vector space is equal to $L$. It follows that $C\subseteq L$. On the other hand $L$ is contained in its centralizer $C$, since $L$ is commutative. We conclude that $C=L$ as claimed. This implies that $K\subseteq L$, since $K$ is the center of $\uM_n(K)$. 
\end{proof}

We denote by $\dim B$ the relative dimension of an abelian scheme $B$ over a connected scheme $S$. The first three statements in the following lemma are (consequences of)  results of Ribet~\cite{ribet:gl2} for endomorphism algebras of abelian varieties over $\QQ$ of $\GL_2$-type.


\begin{lemma}\label{lem:gl2avstructure}
Assume that $S$ is a connected Dedekind scheme with function field $\QQ$, and suppose that $A$ is an abelian scheme over $S$ of $\GL_2$-type. Then the following holds.
\begin{itemize}
\item[(i)] There exist $n\in\ZZ_{\geq 1}$ and an abelian scheme $B$ over $S$ with the following property: $(*)$  $A$ is isogenous  to $B^n$ and  the generic fiber of $B$  is simple.
\item[(ii)] Let $L/\QQ$ be a number field of degree $g$ which is a subring of $\End^0(A)$, and let $B$ be an abelian scheme over $S$ with $(*)$. Then the $\QQ$-algebra $\End^0(B)$ is isomorphic to a subfield $K$ of $L$ such that $K/\QQ$ has degree $\dim B$ and $L/K$ has  degree $n$. 
\item[(iii)] If $L$ is as in (ii), then $L$ is the centralizer of $L$ inside $\End^0(A)$.
\item[(iv)]There is an abelian scheme $B$ over $S$ with $(*)$ such that $\End(B)$ is Dedekind.
\end{itemize}
\end{lemma}
\begin{proof}
We first prove  (i). The generic fiber $A_\QQ$ of $A$ is an abelian variety over $\QQ$ of $\GL_2$-type by \eqref{eq:neron}. Hence \cite[Thm 2.1]{ribet:gl2} gives $n\in\ZZ_{\geq 1}$ together with a simple abelian variety $B_\QQ$ over $\QQ$  such that there is an isogeny $f_\QQ:A_\QQ\to B_\QQ^n$. Then, on using that $A_\QQ$ extends to $A$ over $S$,  results in \cite{bolura:neronmodels} imply the following: The direct factor $B_\QQ$ of $B_\QQ^n$  extends to an abelian scheme $B$ over $S$ and $f_\QQ$ extends to an isogeny  $f:A\to B^n$. This proves  (i). 

To show (ii) we take an abelian scheme $B$ over $S$ with $(*)$ and we consider a number field  $L/\QQ$ of degree $g$ which is a subring of $\End^0(A)$. The simple factors of an abelian variety over a field are unique up to isogeny, the endomorphism algebras of isogenous abelian schemes over $S$ are isomorphic by \eqref{def:phiupstar}, and the endomorphism ring of an abelian scheme over $S$ identifies via \eqref{eq:neron} with the endomorphism ring of its generic fiber. Hence the arguments in the proof of \cite[Thm 2.1]{ribet:gl2} give that the $\QQ$-algebra $\End^0(B)$ is isomorphic to a subfield $K$ of $L$ with $[L:K]=n$ and $\dim B=[K:\QQ]$; here we note that our fields $K$ and $L$ are denoted by $F$ and $E$ in  \cite{ribet:gl2} respectively. This shows (ii).


We next prove  (iii). Assertion (i) gives $n\in\ZZ_{\geq 1}$ together with an abelian scheme $B$ over $S$ with $(*)$. In particular $A$ is isogenous to $B^n$ and thus the $\QQ$-algebra  $\End^0(A)$ is isomorphic to $\uM_n\bigl(\End^0(B)\bigl)$ by \eqref{def:phiupstar}. Now, we assume that $L/\QQ$ is a number field of degree $g$ which is a subring of $\End^0(A)$.   Then assertion (ii) provides an isomorphism of $\QQ$-algebras $\iota:\End^0(A)\isomto\uM_n(K)$ for some subfield $K$ of $L$ with $L/K$ of degree $n$. By transport of structure, the subring $\iota(L)$  of $\uM_n(K)$ becomes a $K$-vector space of dimension $n$. Therefore an application of Lemma~\ref{lem:Centralizer} with the fields $K$ and $\iota(L)$ implies (iii).

To show (iv) we use assertion (i) which gives an abelian scheme $B$ over $S$ with $(*)$. Our $A$ is of $\GL_2$-type and thus (ii) shows that $\End^0(B)$ is a number field. Let $\mathcal O'$ be the normalization  of the Noetherian ring $\OL=\End(B)$ in $\End^0(B)$. Then $\OL'$ is a Dedekind ring 
 and it is a maximal order of  $\OL\otimes_\ZZ\QQ=\End^0(B)$. 
Therefore  Lemma~\ref{lem:abtensorproduct}  gives an abelian scheme $C$ over $S$, isogenous to $B$, such that $\End(C)$  is isomorphic to the Dedekind ring $\mathcal O'$. Furthermore  $C$ has property $(*)$, since it is isogenous to $B$ which satisfies $(*)$.  This proves (iv) and thus completes the proof of Lemma~\ref{lem:gl2avstructure}.
\end{proof}

In Lemma~\ref{lem:gl2avstructure} one can not drop the assumption that the function field of $S$ is $\QQ$.

\section{Par\v{s}in constructions: Forgetting extra structures}\label{sec:parsin}


\noindent Roughly speaking a Par\v{s}in construction for a Diophantine set is a finite map from this set into a certain set of integral points of a moduli scheme of abelian varieties. Faltings' proof \cite{faltings:finiteness} of the Mordell conjecture uses the original Par{\v{s}}in construction~\cite{parshin:construction} for rational points of  curves of genus at least two. In \cite{rvk:intpointsmodellhms} the natural forgetful map was used as a Par\v{s}in construction in order to prove explicit bounds for the number and the height of integral points on any moduli scheme of elliptic curves. 

In the present paper we use the natural forgetful map of a Hilbert moduli scheme as a Par\v{s}in construction. We now introduce some notation and then we decompose the forgetful map in a way which will be very useful for the explicit study of its fibers. 

\subsection{Par\v{s}in construction for Hilbert moduli schemes} 
Let $Y$ be a scheme and let $\mathcal M$ be a Hilbert moduli stack. Suppose that $Y=M_\mathcal P$ is a Hilbert moduli scheme of some presheaf $\mathcal P$ on $\mathcal M$. Then we denote by $$\phi:Y\to \absg $$  the natural forgetful map of $Y=M_{\mathcal P}$ defined in \eqref{def:forgetfulmap}, which is induced by forgetting the extra structures. We shall use this map $\phi=\phi_{\mathcal P}$ as a Par\v{s}in construction for the points of $Y$. To study its fibers, we factor $\phi$ into maps which can be described in quite explicit terms. Before we describe our factorization of $\phi$, we introduce some terminology.

\paragraph{Presheaves.} Suppose that $\mathcal M=\mathcal M^I$ is the Hilbert moduli stack associated to some $L$, $I$ and $I_+$ as in Section~\ref{sec:hms}. In particular $g$ is the degree of $L/\QQ$. For any positive $d\in\ZZ$, we denote by $\mathcal A_{g,d}$ the moduli stack over $\ZZ$ parametrizing abelian schemes of relative dimension $g$  with a polarization of degree $d$. 
Let $M$ and $A_{g,d}$ be the presheaves on $\sch$ which send a scheme $T$ to the set of isomorphism classes of objects of the small categories $\mathcal M(T)$ and $\mathcal A_{g,d}(T)$ respectively. 
The sheafifications of the presheaves $M$ and $A_{g,d}$ on $\sch$ with the \'etale topology identify with the coarse moduli spaces of the separated finite type stacks $\mathcal M/\ZZ$ and $\mathcal A_{g,d}/\ZZ$ respectively, see \cite[Rem 3.19]{lamo:stacks}. Let
\begin{equation}\label{eq:forgetfulpolav}
\phi_\lambda:A_{g,d}\to \absg
\end{equation}
be the morphism of presheaves on $\sch$ which is induced by forgetting the polarization, where $\absg$ is as in \eqref{def:forgetfulmap}.  Narasimhan--Nori~\cite[Thm 1.1]{nano:polarizations} proved that  $\psi$ is finite over any (algebraically closed) field. 
However the fibers of $\phi_\lambda$ are very complicated,  and it seems to be difficult to explicitly control these fibers even over $\QQ$ or $\CC$. 
In view of this we do not factor our Par\v{s}in construction $\phi$ through $\phi_\lambda$. Instead we work with the morphism 
\begin{equation}\label{def:forgetpol}
\phi_\varphi:M\to \abomult
\end{equation}
of presheaves on $\sch$ which is induced by forgetting the $I$-polarization $\varphi$ of an object $(A,\iota,\varphi)$ of $\mathcal M$. Here $\abomult$ denotes the presheaf on $\sch$ which sends a scheme $T$ to the set of isomorphism classes of $\OL$-abelian schemes over $T$ of relative dimension $g$, where $\OL$ is the ring of integers of $L$. Now, the advantage is that the fibers of $\phi_\varphi$ are much simpler than those of $\phi_\lambda$. The main reason for this is that $I$-polarizations are compatible with the involved $\OL$-structure. In fact we shall show in Section~\ref{sec:polarizations} that this compatibility allows to explicitly control the fibers of $\phi_\varphi$ by working with orders which are all commutative.

\paragraph{Decomposition of the Par\v{s}in construction.}
We are now ready to decompose the Par\v{s}in construction $\phi:Y\to\absg$  of the Hilbert moduli scheme $Y=M_{\mathcal P}$. 
In view of Lemma~\ref{lem:moduli}, forgetting the involved $\mathcal P$-level structure $\alpha$ induces a morphism
\begin{equation}\label{def:forgetlvl}
\phi_\alpha: Y\to \hilbmod   
\end{equation}
of presheaves on $\sch$. 
Then we  map $\hilbmod$ to $\abomult$ via the morphism $\phi_\varphi:\hilbmod\to \abomult$ defined in \eqref{def:forgetpol}. Next, we consider the natural forgetful map $\abomult\to \absg$ and we denote by $(A,\iota)$ an $\OL$-abelian scheme of relative dimension $g$. Tensoring the ring morphism $\iota:\OL\to \End(A)$ with $\QQ$ gives an embedding $\iota\otimes\QQ$ of $L$ into $\End^0(A)$ and thus $A$ is of $\GL_2$-type. Hence forgetting the $\OL$-module structure $\iota$ induces a morphism
\begin{equation}\label{def:forgetendo}
\phi_\iota: \abomult\to M_{\GL_2,g}\hookrightarrow \absg
\end{equation}
of presheaves on $\sch$, 
 where $M_{\GL_2,g}$ is the presheaf on $\sch$ which we defined in \eqref{def:gl2functor}. Here we may and do naturally identify $M_{\GL_2,g}$ with the subfunctor  $M_{\GL_2,g}\hookrightarrow\absg$ formed by those abelian schemes which are of product $\GL_2$-type. Finally we observe that the Par\v{s}in construction $\phi:Y\to \absg$ decomposes into the following morphisms 
\begin{equation}\label{eq:forgetfuldecomp}
\phi: Y\to^{\phi_\alpha} \hilbmod\to^{\phi_\varphi} \abomult\to^{\phi_\iota} M_{\GL_2,g}\hookrightarrow \absg.
\end{equation}
In the coming sections, we study in detail the fibers of the morphisms $\phi_\alpha$, $\phi_\varphi$ and $\phi_\iota$. In particular we explicitly control the fibers over any open subscheme  of $\spec(\ZZ)$.





\section{Effective Shafarevich conjecture}\label{sec:es}

In this section we review various explicit finiteness results obtained in \cite{rvk:modularhms}, including the effective Shafarevich conjecture for abelian varieties of $\GL_2$-type.

Let $S$ be a nonempty open subscheme of $\spec(\ZZ)$ and let $g\geq 1$ be an integer. Shafarevich conjectured that the set $\absg(S)$ of isomorphism classes of abelian schemes over $S$ of relative dimension $g$ is finite.  Faltings~\cite{faltings:finiteness} proved Shafarevich's conjecture for polarized abelian schemes, and  Zarhin~\cite{zarhin:shafarevich} showed that one can get rid of the polarizations.

\paragraph{Height.}  To state the effective Shafarevich conjecture~$\ces$, we denote by $h_F(A)$ the stable Faltings height of (the generic fiber of) an arbitrary abelian scheme $A$ over $S$ of relative dimension $g$. We refer to Section~\ref{sec:hms} for the definition of $h_F$.

\vspace{0.3cm}
\noindent{\bf Conjecture~$\ces$.}
\emph{There exists an effective constant $c$, depending only on $S$ and $g$, such that any abelian scheme $A$  over $S$ of relative dimension $g$ satisfies $h_F(A)\leq c.$}
\vspace{0.3cm}

\noindent This conjecture has interesting Diophantine applications. For example, it implies the effective Mordell conjecture for curves of genus at least two defined over arbitrary number fields; see \cite[Prop 9.1]{rvk:modularhms} which uses the explicit Kodaira construction of R\'emond~\cite{remond:construction}.  Conjecture~$\ces$ was established in \cite{rvk:modularhms} for abelian schemes over $S$ of product $\GL_2$-type; the proof combines Faltings' method \cite{faltings:finiteness} with Serre's modularity conjecture~\cite{khwi:serre}, isogeny estimates in \cite{faltings:finiteness,raynaud:abelianisogenies} or \cite{mawu:abelianisogenies,gare:isogenies} and explicit results based on Arakelov theory~\cite{bost:lowerbound,javanpeykar:belyi}. More precisely, if $A$ is an abelian scheme over $S$ of relative dimension $g$ which is of product $\GL_2$-type, then  \cite[Thm 9.2]{rvk:modularhms} gives
\begin{equation}\label{eq:esgl2}
h_F(A)\leq (3g)^{144g}N_S^{24}.
\end{equation}
Here $N_S=\prod p$ is the product of all rational primes $p$ not in $S$. The function field of $S$ is $\QQ$. Hence an abelian scheme $A$ over $S$ of relative dimension $g$ is of product $\GL_2$-type if and only if $\End^0(A)$ contains a commutative semi-simple $\QQ$-subalgebra of degree $g$.

\paragraph{Number of isomorphism classes.}
Further, the set $M_{\GL_2,g}(S)$  of isomorphism classes of abelian schemes over $S$ of relative dimension $g$ which are of product $\GL_2$-type satisfies 
\begin{equation}\label{eq:qesgl2}
|M_{\GL_2,g}(S)|\leq (14g)^{(9g)^6}N_S^{(18g)^4}.
\end{equation}
This bound was established in \cite[Thm 9.6]{rvk:modularhms} by using a variant of the proof of \eqref{eq:esgl2}. 


\paragraph{Uniform isogeny estimates.} The proofs of \eqref{eq:esgl2} and \eqref{eq:qesgl2} use among other things the most recent version, due to Gaudron--R\'emond~\cite{gare:isogenies}, of the Masser--W\"ustholz isogeny estimates \cite{mawu:abelianisogenies,mawu:factorization} for abelian varieties over number fields. These isogeny estimates involve the height $h_F$ and on combining them with the height bound \eqref{eq:esgl2} one obtains the following result (see \cite[Cor 9.5]{rvk:modularhms}): If $A$ and $A'$ are isogenous abelian schemes over $S$ of relative dimension $g$ which are of product $\GL_2$-type, then there exists an isogeny $\psi:A\to A'$ whose degree $\deg(\psi)$  is bounded by 
\begin{equation}\label{eq:mindegboundshaf}
(14g)^{(12g)^5}N_S^{(37g)^3}.
\end{equation} 
We point out that  the above estimate is uniform in $A$ in the sense that it only depends on $S$ and $g$.  This will be crucial for our proofs given in Section~\ref{sec:endo} where we use \eqref{eq:mindegboundshaf} in order to reduce to the key case of Theorem~\ref{thm:endobound}. 

\section{Level structures}\label{sec:levelstructures}

Let $\mathcal M$ be a Hilbert moduli stack, let $\mathcal P$ be a moduli problem on $\mathcal M$ and let $S$ be a scheme.  In this section we prove basic properties of the quantity $|\mathcal P|_S$, defined in \eqref{def:maxlvl}, which is the maximal number  of $\mathcal P$-level structures on any object of $\mathcal M(S)$.

Let $Y$ be a scheme. Suppose that $Y=M_{\mathcal P}$ is a Hilbert moduli scheme of $\mathcal P$ and  let $\phi_\alpha:Y\to \hilbmod$ be the morphism \eqref{def:forgetlvl} of presheaves on $\sch$ induced by forgetting $\mathcal P$-level structures. The following result is a formal generalization to arbitrary Hilbert moduli schemes of the lemma stated in \cite[Lem 3.1]{rvk:intpointsmodellhms} for moduli schemes of elliptic curves.

\begin{lemma}\label{lem:degforgetlvl}
For any scheme $S$, the degree of $\phi_\alpha(S)$ is at most $|\mathcal P|_S$.
\end{lemma}
\begin{proof}
Lemma~\ref{lem:moduli} identifies the set $Y(S)$ with the set of isomorphism classes of pairs $(x,\alpha)$ with $x\in \mathcal M(S)$ and $\alpha\in\mathcal P(x)$. Further we recall that $M(S)$ is the set of isomorphism classes of objects of the small category $\mathcal M(S)$ and that $|\mathcal P|_S=\sup |\mathcal P(x)|$ with the supremum taken over all objects $x\in\mathcal M(S)$. Now, we proceed as in the proof of \cite[Lem 3.1]{rvk:intpointsmodellhms} and we suppose that $\{[(x_i ,\alpha_i)]\}$ is the fiber of $\phi_\alpha(S)$ over a point $[x]$ in $\hilbmod(S)$. Then all $x_i$ are isomorphic to $x$ in $\mathcal M(S)$. Therefore, after applying suitable isomorphisms in $\mathcal M(S)$, we may and do assume that all $x_i=x$ coincide. But then we conclude that $|\phi_\alpha(S)^{-1}([x])|\leq |\mathcal P(x)|\leq |\mathcal P|_S$, which proves Lemma~\ref{lem:degforgetlvl}. 
\end{proof}

The degree of $\phi_\alpha(S)$ is not necessarily finite. However, for those $\mathcal P$ of interest in arithmetic one can usually combine Lemma~\ref{lem:degforgetlvl} with geometric arguments to show that the degree of $\phi_\alpha(S)$ is finite and can be controlled explicitly. This normally involves a reduction to the geometric case when $S$ is the spectrum of an algebraically closed field, and this reduction can be done by using the following lemma. 

\begin{lemma}\label{lem:geomlvlred}If, moreover, $Y=M_{\mathcal P}$ is a variety over $\ZZ$,  then for each integral scheme $B$, any field $F$ containing the function field of $B$ and  all nonempty open  $T\subseteq S\subseteq B$, it holds $$\deg \phi_\alpha(S)\leq |\mathcal P|_S\leq |\mathcal P|_T\leq |\mathcal P|_F.$$
\end{lemma}
\begin{proof}
Lemma~\ref{lem:degforgetlvl} gives the first inequality. To prove the remaining inequalities, we freely use the terminology and results of \cite[$\mathsection 4$]{demu:stacks}. By assumption $Y$ is a Hilbert moduli scheme of $\mathcal P$. Hence there exists $y\in\mathcal M(Y)$ which represents the presheaf $\mathcal P$ on $\mathcal M$. Let $S$ be an arbitrary scheme, let $x\in\mathcal M(S)$ and consider the presheaf $Y_x$ on $\sch/S$ which sends an $S$-scheme $\varphi:S'\to S$ to the set $\mathcal P(\varphi^*x)$. 
The diagonal of the algebraic stack $\mathcal M/\ZZ$ is representable 
 and hence $Y_x$ is represented by the $S$-scheme $Y_x=\textnormal{Isom}(U,p_1^*x,p_2^*y)$,
where $p_1:U\to S$ and $p_2:U\to Y$ are the natural projections of $U=S\times_\ZZ Y$.  In particular, for any $S$-scheme $\varphi:S'\to S$ the pullback $x'=\varphi^*x$ of $x$ lies in $\mathcal M(S')$ and satisfies
\begin{equation}\label{eq:yxcomp}
Y_x(S')\cong \mathcal P(x').
\end{equation}
The structure morphism  of the $S$-scheme $Y_x$ factors as $p_1p_U$ for $p_U:Y_x\to U$ the projection. The morphism $p_1:Y_S\to S$ is separated finite type, since $Y$ is a variety over $\ZZ$.  Further, $p_U:Y_x\to U$ is a pullback of the diagonal $\mathcal M\to \mathcal M\times_\ZZ \mathcal M$ and this diagonal is proper since $\mathcal M/\ZZ$ is separated. 
It follows that $Y_x$ is a variety over $S$.

From now on we assume that $T$, $S$, $B$ and $F$ are as in the statement of the lemma. We first show that $|\mathcal P|_S\leq |\mathcal P|_T$. To obtain an inclusion $Y_x(S)\hookrightarrow Y_x(T),$  
we use that $Y_x$ is a separated $S$-scheme which assures that the equalizer of two $S$-scheme morphisms $S'\to Y_x$ is a closed subscheme of $S'$.  Further, any nonempty open subscheme of an irreducible scheme is dense and again irreducible, and $B$ is irreducible by assumption. Thus restriction from the irreducible $S$ to the open $T\subseteq S$ gives an inclusion $Y_x(S)\hookrightarrow Y_x(T)$ as desired.  Then applications of \eqref{eq:yxcomp} with $S'=S$ and $S'=T$ give that $|\mathcal P(x)|=|Y_x(S)|$ is at most $|Y_x(T)|=|\mathcal P(x')|$. In other words, for any object $x\in\mathcal M(S)$ there exists an object $x'\in\mathcal M(T)$ such that $|\mathcal P(x)|\leq |\mathcal P(x')|$. This implies that $|\mathcal P|_S\leq |\mathcal P|_T$ as claimed.  

 The proof of the inequality $|\mathcal P|_T\leq |\mathcal P|_F$ is essentially the same. Indeed an application  of the above  arguments, with $T$ and $F$ in place of $S$ and $T$ respectively, gives for any object $x\in\mathcal M(T)$ that $|\mathcal P(x)|=|Y_x(T)|$ is at most $|Y_x(F)|=|\mathcal P(x')|$. Here $x'=\varphi^*x$ lies in $\mathcal M(F)$ where $\varphi:\spec(F)\to T$ is induced by the inclusion $k(T)=k(B)\subseteq F$, and to see that  $Y_x(T)\to Y_x(F)$ is injective we now use that the image of $\varphi$ is the generic point of $T$ which is dense in the irreducible scheme $T$. 
This completes the proof of the lemma.
\end{proof}

\section{Polarizations}\label{sec:polarizations}

Let $L$ be a totally real number field of degree $g$ over $\QQ$ with ring of integers $\OL$. As in Section~\ref{sec:hms}, we denote by  $I$  a nonzero finitely generated $\OL$-submodule of $L$ together with a positivity notion $I_+$. In this section we study equivalence classes of $I$-polarizations on $\OL$-abelian schemes of relative dimension $g$. In particular, we uniformly control the number of such equivalence classes over certain base schemes by combining formal computations with Lemma~\ref{lem:gl2avstructure} and Dirichlet's unit theorem for orders in number fields. 

Let $S$ be a scheme and let $A$ be an $\OL$-abelian scheme over $S$ of relative dimension $g$. We recall that an $I$-polarization  on $A$ is an $\OL$-module morphism from $I$ to the $\OL$-module $\Hom_\OL(A,A^\vee)^{\textnormal{sym}}$, which maps $I_+$ to polarizations  and which  induces an isomorphism $I \otimes_{\OL} A \isomto A^\vee$. Let $\Phi$ be the set of $I$-polarizations on $A$. We say that  $\varphi,\varphi'\in\Phi$ are equivalent  if there exists an automorphism $\sigma$ of the $\OL$-abelian scheme $A$ such that $\varphi(\lambda)=\sigma^\vee\varphi'(\lambda)\sigma$ for all $\lambda\in I$.  The following result bounds the number of equivalence classes.

\begin{proposition}\label{prop:polbound}
Suppose that $S$ is a connected Dedekind scheme with function field $\QQ$. Then the set $\Phi$ decomposes into at most $2^g$ equivalence classes.  
\end{proposition}

Let $\mathcal M$ be the Hilbert moduli stack associated to $L$, $I$ and $I_+$. Before we give a proof of the above proposition, we deduce a corollary which explicitly controls the degree of the natural forgetful map $\phi_\varphi \colon \hilbmod \to \abomult$ defined in \eqref{def:forgetpol} with respect to $\mathcal M$.

\begin{corollary}\label{cor:forgetpol}
For any $S$  as in Proposition~\ref{prop:polbound}, the degree of $\phi_\varphi(S)$ is at most $2^g$.
\end{corollary}
\begin{proof}
To prove the statement, we may and do assume that $\hilbmod(S)$ is nonempty. Take a class $[(A, \iota,\varphi)] \in \hilbmod(S)$ and let $\Phi$ be the set of $I$-polarizations on $(A,\iota)$. We observe that two $I$-polarizations  $\varphi,\varphi'\in \Phi$ are equivalent if and only if $(A,\iota,\varphi)$ and $(A,\iota,\varphi')$ define the same class in $\hilbmod(S)$. Hence mapping the class of $(A,\iota,\varphi)$ in $\hilbmod(S)$ to the equivalence class of $\varphi$ induces a bijection between the fiber of $\phi_\varphi$ over the class of $(A,\iota)$ in $\abomult(S)$ and the set $\Phi/\sim$. 
Thus the bound for $|\Phi/\sim|$ in Proposition~\ref{prop:polbound} implies the corollary.
\end{proof}

In what follows in this section we write $\otimes=\otimes_\OL$ for simplicity.
To prove Proposition~\ref{prop:polbound}, we compute the set $\Phi/\sim$ in terms of certain isomorphisms $A\isomto A^\vee\otimes I^{-1}$. Multiplication defines an isomorphism $\mu:I^{-1}\otimes I\isomto \OL$ of $\OL$-modules, since $I$ is invertible. Further, in the category of $\OL$-modules it holds: If $X,Y,Z$ are $\OL$-modules, then mapping any morphism $\varphi:X\to \Hom(Y,Z)$  to the morphism $X\otimes Y\to Z$ defined by $x\otimes y\mapsto \varphi(x)(y)$ gives a bijection $\Hom(X,\Hom(Y,Z))\isomto \Hom(X\otimes Y,Z)$ which is functorial. 
Thus, on using the isomorphism $I^{-1}\otimes I\otimes A\isomto A$ induced by $\mu$ and tensoring with $I^{-1}$, we obtain a bijection 
$$\tau:\Hom(I,\Hom_{\mathcal C}(A,A^\vee))\isomto \Hom_{\mathcal C}(A,I^{-1}\otimes A^\vee),$$
where $\mathcal C$ denotes the category of $\OL$-abelian schemes over $S$. To compute the set $\tau(\Phi)$  we take $\lambda\in I$. If $T$ is an $S$-scheme, then sending $j\otimes x\in I^{-1}\otimes A^{\vee}(T)$ to the element $j\lambda \otimes x$ in $\OL\otimes A^{\vee}(T)\cong A^{\vee}(T)$ defines a morphism $m_\lambda: I^{-1}\otimes A^{\vee}\to  A^{\vee}$ of $\OL$-abelian schemes over $S$.
We denote by $\Psi$ the set of isomorphisms $\psi: A\isomto I^{-1}\otimes A^\vee$ of $\OL$-abelian schemes over $S$ such that for each $\lambda\in I$ the morphism $m_\lambda \psi:A\to A^\vee$ is symmetric and moreover a polarization if $\lambda\in I_+$.  
The following formal lemma computes $\tau(\Phi)$.

\begin{lemma}\label{lem:polformal}
It holds $\tau(\Phi)=\Psi$.  
\end{lemma}

\begin{proof}

Suppose that $\varphi:I\to \Hom_{\mathcal C}(A, A^\vee)$ is a morphism of $\OL$-modules and denote by $\rho:I \otimes A \to A^\vee$ the induced morphism. Then the construction of $\tau$ shows that $\psi=\tau(\varphi)$ is the morphism $A\isomto I^{-1}\otimes A^\vee$ given by $\psi=( I^{-1}\otimes \rho)\mu_A^{-1}$, where  $\mu_A:I^{-1}\otimes I\otimes A\isomto A$ is induced by $\mu:I^{-1}\otimes I \isomto \OL$. We claim that each $\lambda\in I$ satisfies
\begin{equation}\label{eq:mlambda}
m_\lambda\psi=\varphi(\lambda).
\end{equation}
To verify this claim, we consider an $S$-scheme $T$ and we take $j\otimes i\otimes x$ in $I^{-1}\otimes I\otimes A(T)$. On using that $j\lambda\in\OL$ and that $\varphi(i):A(T)\to A^\vee(T)$ is a morphism of $\OL$-modules, we see that $m_\lambda\psi\mu_A(j\otimes i\otimes x)=\varphi(i)(j\lambda  x)$. Further, we compute that $\varphi(i)(j\lambda  x)=\varphi(\lambda)\mu_A(j\otimes i\otimes x)$ since $\varphi$ is a morphism of $\OL$-modules and since $ji\in \mathcal O$. This shows that $m_\lambda\psi\mu_A=\varphi(\lambda)\mu_A$, and then applying the inverse $\mu_A^{-1}$ of the isomorphism $\mu_A$ implies our claim \eqref{eq:mlambda}. 

Now, we see that $\tau(\Phi)$ is contained in $\Psi$. Indeed, if $\varphi\in \Phi$ then \eqref{eq:mlambda} implies that $\tau(\varphi)\in \Psi$ by using the properties of $\varphi\in\Phi$. 
To show the converse, that $\Psi$ is contained in $\tau(\Phi)$, we may and do assume that there exists $\psi\in\Psi$. Then the construction of $m_\lambda$ shows that $\lambda\mapsto m_\lambda\psi$ defines a morphism $\varphi:I\to \Hom_{\mathcal C}(A, A^\vee)$ of $\OL$-modules.  Furthermore, on using again that $ji\in \OL$ and that $\psi$ is a morphism of $\OL$-abelian schemes, we directly compute that $\tau(\varphi)\mu_A=\psi\mu_A$ and hence we obtain that $\tau(\varphi)=\psi$. Then an application of \eqref{eq:mlambda} with $\varphi$ and $\psi=\tau(\varphi)$ gives that $m_\lambda\psi=\varphi(\lambda)$ for each $\lambda\in I$, and thus $\varphi\in \Phi$ since $\psi$ lies in $\Psi$.  
This proves that $\Psi\subseteq\tau(\Phi)$, and we conclude that $\tau(\Phi)=\Psi$ as desired.\end{proof}

On using the canonical bijection $\tau:\Phi\isomto \Psi$ and the above defined equivalence relation on $\Phi$, we obtain an equivalence relation on $\Psi$ by transport of structure. The following result uses Dirichlet's unit theorem to bound the number of equivalence classes.

\begin{lemma}\label{lem:psipolbound}
Suppose that $S$ is a connected Dedekind scheme with function field $\QQ$. Then the set $\Psi$ decomposes into at most $2^g$ equivalence classes.
\end{lemma}
\begin{proof}
To estimate the cardinality of $\Psi/{\sim}$,  we may and do assume that there exists an isomorphism  $\psi_0:A\isomto I^{-1}\otimes A^\vee$ which lies in $\Psi$.  If $\iota:\OL\to \End(A)$ is the $\OL$-module structure of the $\OL$-abelian scheme $A$, then $\iota\otimes_\ZZ\QQ$ embeds $L$ into $\End^0(A)$ and the centralizer of $L$ in $\End^0(A)$ contains $\End^0_\mathcal C(A)=\End_\mathcal C(A)\otimes_\ZZ\QQ$. Further, $\OL$ is commutative and our $A/S$ satisfies the assumptions of Lemma~\ref{lem:gl2avstructure}. Thus Lemma~\ref{lem:gl2avstructure}~(iii) implies that $\End_\mathcal C^0(A)\cong L$. Hence $\End_{\mathcal C}(A)$ identifies with an order $\Gamma$ of $L$ whose group of units $\Gamma^\times$ identifies with $\Aut_\mathcal C(A)$, and $\psi \mapsto \psi_0^{-1}  \psi$ defines an injective map
$
j \colon \Psi \hookrightarrow \Gamma^\times$. Next we consider the subgroup $\Gamma^{\times 2}= \{\gamma^2\ |\ \gamma \in \Gamma^\times\}$ of $\Gamma^\times$. We claim that $j$ induces an embedding 
\begin{equation}\label{eq:InclusionModSqaures}
\Psi/{\sim} \hookrightarrow \Gamma^\times / \Gamma^{\times 2}.
\end{equation} 
To prove this claim, we suppose that $\psi,\psi'\in\Psi$ and we denote by $\varphi,\varphi'\in \Phi$ their preimages under $\tau$ respectively. Then $\psi\sim\psi'$  if and only if there exists  $\sigma \in \Aut_\mathcal C(A)$ with $
\varphi' =\sigma^\vee \varphi \sigma
$. Moreover, on applying the formal arguments of the proof of Lemma~\ref{lem:polformal}, we compute that $
\varphi' =\sigma^\vee \varphi \sigma
$ if and only if  
$
\psi' = (I^{-1}\otimes\sigma^\vee) \psi \sigma
$.
The latter equality is equivalent to 
$
\psi_0^{-1} \psi' = \psi_0^{-1} ( I^{-1}\otimes \sigma^\vee) \psi_0 \psi_0^{-1}  \psi  \sigma,
$
which in turn is equivalent to
$
j(\psi') = \sigma^\star j(\psi) \sigma  \in \Gamma^\times
$ 
for
$$f^\star = \psi_0^{-1}  (I^{-1}\otimes f^\vee)  \psi_0,\quad f \in \End_{\mathcal C}(A).$$
Let $\lambda\in I_+$ and denote by $^*$ the usual Rosati involution on $\End^0(A)$ associated to the polarization $\varphi(\lambda)$ of $A$, where $\varphi=\tau^{-1}(\psi_0)$ lies in $\Phi$ by Lemma~\ref{lem:polformal}. Then our $^\star$ coincides on $\Gamma$ with the restriction of the Rosati involution $^*$ to $\End_\mathcal C^0(A)$. Indeed, on using that any $f$ in $\End_\mathcal C(A)\cong \Gamma$ is a morphism of $\OL$-abelian schemes, we compute $f^\vee m_\lambda=m_\lambda (I^{-1}\otimes f^\vee)$ and then we deduce that $f^*=\varphi(\lambda)^{-1}f^\vee \varphi(\lambda)=f^\star$ inside $\End_\mathcal C^0(A)$ since $m_\lambda\psi_0=\varphi(\lambda)$ by \eqref{eq:mlambda}.  Our number field $L$ is totally real, and hence it follows for example from \cite[p.5]{lang:cm} that the positive involution
$^*$ acts as the identity on $L$ inside $\End^0(A)\subseteq \End^0(A_\CC)$. 
Thus, on exploiting that $\Gamma$ is commutative, we conclude that $\psi\sim\psi'$ if and only if $j(\psi') = \sigma^2 j(\psi)$ in  $\Gamma^\times$
for some $\sigma\in \Gamma^\times$.
This proves our claim in \eqref{eq:InclusionModSqaures}.  Dirichlet's unit theorem~\cite[p.81]{neukirch:ant} gives that the free part of $\Gamma^\times$ has rank $g-1$, since $L$ is totally real. Therefore $|\Gamma^\times / \Gamma^{\times 2}|$ is at most $|\mu(L)|\cdot 2^{g-1}$, where $\mu(L)=\{1,-1\}$ are the only roots of unity in $L$. 
It follows that $|\Gamma^\times / \Gamma^{\times 2}| \leq 2^g$ and then \eqref{eq:InclusionModSqaures} implies the lemma. We mention that our proof gives more generally a bound for any connected Dedekind scheme $S$ whose function field $k$ is algebraic over $\QQ$. Indeed  $\End^0_{\mathcal C}(A)$ embeds into $F=\End^0_\mathcal C(A_{\CC})$ via $k\to \CC$, a complex Lie algebra argument shows that either $F=L$ or $F/L$ is a quadratic CM extension, and the above arguments embed $\Psi/\sim$ into either $\Gamma^\times/\Gamma^{\times 2}$ or into $\Gamma^\times/N(\Gamma^{\times})$ for $N=N_{F/L}$ the norm; here we used that $^\star=^{ \ *}$ acts as complex conjugation on $\Gamma$.
\end{proof}

Finally, on combining the above lemmas we obtain Proposition~\ref{prop:polbound}.

\section{Endomorphism structures}\label{sec:endo}

 Let $\order$ be an order of an arbitrary number field $L$ of degree $g=[L:\QQ]$, and let $S$ be an open subscheme of $\spec(\ZZ)$.  In this section we study $\order$-structures on any abelian scheme $A$ over $S$ of relative dimension $g$. In particular, we explicitly control the number of such $\order$-structures by using inter alia isogeny estimates based on transcendence.

We start with a definition. Let $R$ be a not necessarily commutative ring and let $\Hom(\order,R)$ be the set of ring morphisms $\order\to R$. The unit group $R^\times$ of $R$ acts on  $\Hom(\order,R)$  by conjugation, that is $(r,f)\mapsto (\gamma\mapsto rf(\gamma)r^{-1})$ for $r\in R^\times$ and $f:\order\to R$ a ring morphism. We call an element $\rho$ of $\Hom(\order,R)/R^\times$ a $\order$-structure on $R$, and if $R$ is the endomorphism ring of an abelian scheme $A$ then we call $\rho$ a $\order$-structure on $A$.

\begin{theorem}\label{thm:endobound}
There are at most $d^{(2g+1)g}N(\idf)^{g+1} hl$  distinct $\order$-structures on any abelian scheme over $S$ of relative dimension $g$, where $d= (14g)^{(12g)^5}N_S^{(37g)^3}$. 
\end{theorem}

Here $l$ denotes the degree over $\QQ$ of a normal closure of $L/\QQ$, and $h$ is the class number of the ring of integers of $L$. Further $N_S=\prod p$ is the product of all rational primes $p$ not in $S$, and $N(\mathfrak f)=N_{L/\QQ}(\mathfrak f)$ is the norm of the conductor ideal $\mathfrak f$ of $\order$ (see \cite[p.79]{neukirch:ant}).

Let $A_\QQ$ be an abelian variety over $\QQ$ of dimension $g$, and let $d(A_\QQ)$ be the `minimal' isogeny degree of $A_\QQ$ defined in Remark~\ref{rem:abendo}. We  shall prove in \eqref{eq:abendomindeg} that there are at most $d(A_\QQ)^{(2g+1)g}N(\idf)^{g+1}hl$ distinct $\order$-structures on $\End(A_\QQ)$.  This together with \eqref{eq:neron} provides a more precise version of Theorem~\ref{thm:endobound} for any given abelian scheme $A$ over $S$.  On the other hand, the bound in Theorem~\ref{thm:endobound} is uniform for all abelian schemes over $S$ of relative dimension $g$. {This additional uniformity is crucial for  the following corollary which controls the degree of the natural forgetful map $\phi_\iota:\abomult\to \absg$ defined in \eqref{def:forgetendo}.  

\begin{corollary}\label{cor:forgetendo}
If $S$ is an open subscheme of $\spec(\ZZ)$, then $\deg \phi_\iota(S)\leq d^{(2g+1)g}hl$.
\end{corollary}
\begin{proof}
We denote by $\OL$ the ring of integers of $L$. To bound the degree of $f=\phi_\iota(S)$, we may and do assume that there exists a point $P=[A]$ in $\absg(S)$. Let $\iota$ and $\iota'$ be ring morphisms $\mathcal O\to \End(A)$. The $\OL$-abelian schemes $(A,\iota)$ and $(A,\iota')$ are isomorphic if and only if there exists $r\in \Aut(A)$ with $r\iota(\gamma)r^{-1}=\iota'(\gamma)$ for all $\gamma\in \mathcal O$. Thus $\iota$ and $\iota'$ define the same $\OL$-structure on $A$ if and only if $(A,\iota)$ and $(A,\iota')$ coincide in $\abomult(S)$.  Hence, after applying suitable isomorphisms of $\OL$-abelian schemes over $S$, we see that the fiber $f^{-1}(P)$ over $P$ identifies with the set of pairs $(A,\rho)$ where $\rho$ is an $\OL$-structure on $A$. This shows that  $|f^{-1}(P)|$ coincides with  the number $n(A)$ of $\OL$-structures on $A$. We deduce  
\begin{equation}\label{eq:degsupbound}
\deg(f)\leq \sup_{A/S}n(A)
\end{equation}
with the supremum taken over all abelian schemes $A$ over $S$ of relative dimension $g$. For any such abelian scheme $A$, an application of Theorem~\ref{thm:endobound} with $\order=\OL$ gives an upper bound for $n(A)$ depending only on $g$, $h$, $l$ and $N_S$.  This bound together with \eqref{eq:degsupbound} leads to an estimate for $\deg(f)=\deg(\phi_\iota(S))$ as claimed in Corollary~\ref{cor:forgetendo}.\end{proof}

In the remaining of this section we prove Theorem~\ref{thm:endobound} by using the strategy outlined in the introduction. In Section~\ref{sec:endokeycase} we establish a sharper version of Theorem~\ref{thm:endobound} in the key case. Then  we show in Section~\ref{sec:endoreduction} how to reduce the problem to the key case, and finally we deduce the theorem by putting everything together in Section~\ref{sec:endoproof}. 
\subsection{The key case of Theorem~\ref{thm:endobound}}\label{sec:endokeycase}

We continue our notation and terminology. Let $\order$ be an order of an arbitrary number field $L$ of degree $g=[L:\QQ]$, and let $K\subseteq L$ be a subfield of relative degree $n=[L:K]$.   In this section, we assume throughout that $R=\uM_n(\mathcal O)$ where $\OL$ denotes\footnote{Throughout this section we write $\OL=\OL_K$ in order to simplify notation. We warn the reader that our notation is different in other sections where $\OL$ often denotes the ring of integers of $L$.} the ring of integers of $K$  and we study $\order$-structures on $R$.
In particular, we prove the following result.

\begin{proposition}\label{prop:endoboundmo}
There are at most $N(\mathfrak f)^{g}h(\order)t$ distinct $\order$-structures on $R$.
\end{proposition}
Here $t$ denotes the number of ring morphisms $K\to L$, and $h(\order)=|\Pic(\order)|$ is the class number of $\order$ defined in Section~\ref{sec:fracideals}. Further we recall that  $N(\mathfrak f)$ is the norm of the conductor ideal $\mathfrak f$ of  $\order$. We observe that Proposition~\ref{prop:endoboundmo} is sharper than Theorem~\ref{thm:endobound}  in the key case when the endomorphism ring of the abelian scheme is $\uM_n(\mathcal O)$.  

The strategy of proof of Proposition~\ref{prop:endoboundmo} is as follows. In Section~\ref{sec:compatiblestruct}, we first  decompose the set of $\order$-structures on $R$ into ``compatible'' subsets and then  we identify in Proposition~\ref{prop:bijection} each of these subsets with the set of isomorphism classes of certain $\order$-module structures on $\mathcal O^n$. This allows us in Section~\ref{sec:fracideals} to construct an injective map
\begin{equation}\label{eq:fundendoinjection}
\Hom(\order,R)/R^\times\hookrightarrow \cup_\varphi C_\order
\end{equation}
with the disjoint union taken over all ring morphisms $\varphi:K\to L$. Here $C_\order$ denotes the monoid of (not necessarily invertible) fractional ideals of $\order$. Further, we use an idea of Lenstra  to describe $C_\order=\Pic(\order)\cdot\mathcal I$ as a product of $\Pic(\order)$ with a certain controlled set $\mathcal I$. This description of $C_\order$ then allows us to deduce Proposition~\ref{prop:endoboundmo} from \eqref{eq:fundendoinjection}.

\subsubsection{Compatible morphisms and classes of $\order$-module structures on $\mathcal O^n$}\label{sec:compatiblestruct}

We continue our notation and terminology. Recall that $\order$ is an order of an arbitrary number field $L$ and $\mathcal O$ is the ring of integers of a number field $K\subseteq L$ with $n=[L:K]$.  In this section we first decompose the set of $\order$-structures on $R=\uM_n(\mathcal O)$  into subsets consisting of morphisms which are compatible with some ring morphism $K\to L$. Then we identify each of these subsets with the set of isomorphism  classes of certain $\order$-modules.

\paragraph{Compatible morphisms.} 
Let $\varphi:K\to L$ be a morphism of rings and consider $L$ as a $K$-algebra via $\varphi$. We say that a ring morphism $\rho \colon \order \to R$ is $\varphi$-compatible if tensoring with $\QQ$ induces a morphism $
\rho_\QQ:L\to \uM_n(K)
$
 of $K$-algebras. Here we used the canonical identifications $\order\otimes_\ZZ\QQ=L$ and $R\otimes_\ZZ\QQ= \uM_n(K)$. 
If $\rho:\order\to R$ is a ring morphism which is $\varphi$-compatible and if $r\in R^\times$, then the ring morphism sending $\gamma$ to $r\rho(\gamma) r^{-1}$ is $\varphi$-compatible as well. Therefore the unit group $R^\times$ acts via conjugation on the set
$
\Hom_\varphi(\order, R)
$
of  $\varphi$-compatible ring morphisms $\order\to R$. We obtain the following lemma.
\begin{lemma}\label{lem:decomp}
For each $\rho \in \Hom(\order, R)$ there exists a unique ring morphism $\varphi \colon K \to L$ such that $\rho$ is $\varphi$-compatible. In particular, it holds $$\Hom(\order,R)/R^\times=\cup_{\varphi}\bigl(\Hom_\varphi(\order,R)/R^\times\bigl)$$ with the disjoint union taken over all ring morphisms $\varphi:K\to L$. 
\end{lemma}

\begin{proof}
We take $\rho\in \Hom(\order,R)$. The induced morphism $\rho_\QQ:L\to\uM_n(K)$  is injective since $L$ is a field. Then we obtain a $K$-vector space structure on $L'=\rho_\QQ(L)$ by using the bijection $\rho_\QQ:L\isomto L'$ and the $K$-action on $L$ given by multiplication with elements in $K\subseteq L$.  The dimension of this $K$-vector space $L'$ is $n=[L:K]$. Hence an application of Lemma~\ref{lem:Centralizer} with the subring $L'$ of $\uM_n(K)$ gives that $L'$ is  the centralizer of $L'$ in $\uM_n(K)$. In particular $L'$ contains the field $K$. Let $\varphi:K\to L$ be the restriction to $K\subseteq L'$ of the isomorphism $\rho_\QQ^{-1}:L'\isomto L$. Then we see that our $\rho$ is $\varphi$-compatible. If $\rho$ is in addition $\varphi'$-compatible for some ring morphism $\varphi':K\to L$, then it holds that $\varphi=\varphi'$. Indeed we compute that $\varphi(k)=\rho_\QQ^{-1}(k\rho_\QQ(x))x^{-1}=\varphi'(k)$ for each $k\in K$ and any nonzero $x\in L$.  It follows that $\Hom(\order,R)$ equals $\cup_{\varphi}\bigl(\Hom_\varphi(\order,R)\bigl)$ with the disjoint union taken over all ring morphisms $\varphi:K\to L$. This completes the proof of the lemma. 
\end{proof}

The above  decomposition shows that in order to determine the set of $\order$-structures on $R$ it suffices to know the sets  $\Hom_\varphi(\order, R)/R^\times$. In light of this, we next compute these sets in terms of isomorphism classes of certain $\order$-module structures on $\mathcal O^n$.

\paragraph{Classes of $\order$-module structures on $\mathcal O^n$.} Let $\varphi:K\to L$ be a ring morphism. We now construct a `large' subring $\mathcal O'\subseteq \mathcal O$ such that the restriction of $\varphi$ to $\mathcal O'$ is a ring morphism $\mathcal O'\to\order$.  Suppose that $\{\alpha_i\}$ generates $\mathcal O$ as a $\Z$-module. 
 Then, on using that $\order$ is an order of $L$, we may and do choose $z_i\in\Z$ with $z_i\varphi(\alpha_i)\in\order$. We put $z=\prod z_i$ and we 
let $\mathcal O' = \Z[z\mathcal O]$ be the subring of $\mathcal O$ generated by the subset $z\mathcal O \subseteq \mathcal O$. 
The restriction of $\varphi$ to $\mathcal O'$ is a ring morphism $\mathcal O'\to\order$, which gives any $\order$-module $X$ an $\mathcal O'$-module structure $X|_{\mathcal O'}$. We denote by $\mathcal O^n|_{\mathcal O'}$ the $\mathcal O'$-module given by $\mathcal O^n$ together with the inclusion $\mathcal O'\hookrightarrow \End_\ZZ(\mathcal O^n)$ coming from $\mathcal O'\subseteq \mathcal O\subseteq \End_\ZZ(\mathcal O^n)$. To simplify notation, we write here $\End_\ZZ=\End_{\mathcal C}$ for $\mathcal C$ the category of $\ZZ$-modules. Let
$\iS_{\varphi,z}$ be the set of isomorphism classes $[X]$ of $\order$-modules $X$ such that the $\mathcal O'$-modules $X|_{\mathcal O'}$ and $\mathcal O^n|_{\mathcal O'}$ are isomorphic. 
We write 
\begin{equation}\label{def:isoclassesofgamma}
\iS_{\varphi}=\iS_{\varphi,z}
\end{equation}
since the latter set does not depend on the choice of $\alpha_i$ and $z_i$. 
Indeed, if $z, z'$ are rational integers with $\varphi(z\mathcal O)$ and $\varphi(z'\mathcal O)$  both contained in $\order$, then it turns out that $\iS_{\varphi, z}\cong\iS_{\varphi, z'}$. 
We now identify $\iS_\varphi$ with the set of $\varphi$-compatible $\order$-structures on $R$.

\begin{proposition}\label{prop:bijection}
There exists a bijection $\Hom_\varphi(\order,R)/R^\times\cong \iS_\varphi.$ 
\end{proposition}

\begin{proof}
We explicitly construct a bijection between these two sets. To obtain a map $\Hom_\varphi(\order,R)/R^\times\to \iS_\varphi$, we take a $\varphi$-compatible $\order$-structure $[\rho]$ on $R$ and we give the abelian group $X=\mathcal O^n$ a $\order$-module structure $X_\rho$ via $\rho \colon \order \to R \subseteq \End_{\Z}(X)$. Suppose now that $\rho'$ lies in $[\rho]$. Then there exists $r\in R^\times$ with $r\rho=\rho'r$ and it follows that the morphism $X_{\rho} \to X_{\rho'}$ defined by $x \mapsto rx$ is an isomorphism of $\order$-modules. 
The identity morphism of $\mathcal O^n$ is  an isomorphism of $\mathcal O'$-modules $X_\rho|_{\mathcal O'}\cong\mathcal O^n|_{\mathcal O'}$, since $\rho$ is $\varphi$-compatible. 
Hence the $\order$-isomorphism class $[X_\rho]$ of $X_\rho$ lies in $\iS_\varphi$ and thus we obtain a map
$$\Hom_\varphi(\order, R)/R^\times\to \iS_\varphi, \ \ [\rho]\mapsto [X_\rho].$$
To construct an inverse of this map, we consider $[X]\in \iS_\varphi$. Suppose that the $\order$-module structure on $X$ is given by $\rho \colon \order\to\End_\Z(X)$, and let $\tau \colon X|_{\mathcal O'} \isomto \mathcal O^n|_{\mathcal O'}$ be an isomorphism.  We define $\rho_X \colon \order \to \End_{\Z}(\mathcal O^n)$ by $\gamma \mapsto \tau \rho(\gamma) \tau^{-1}$ for $\gamma\in \order$. It turns out that $\rho_X(\order)$ is contained in  $\End(\mathcal O^n|_{\mathcal  O'})$, since $\varphi(\mathcal O')$ is contained in the commutative ring $\order$ and since $\tau$ is $\mathcal O'$-compatible.  Hence $\rho_X$ maps $\order$ into $\End(\mathcal O^n|_{\mathcal  O'}) \subseteq \End_{\Z}(\mathcal O^n)$. 
Furthermore, on using that $z\mathcal O\subseteq \mathcal O'$ and that $\mathcal O$ is an integral domain, we deduce that $\End(\mathcal O^n|_{\mathcal O'})$ coincides with the endomorphism ring $R$ of the $\mathcal O$-module $\mathcal O^n$. 
Thus  $\rho_X$ is in fact a ring morphism $\order \to R$. Suppose now that  $X'$ lies in the class $[X]$. Then there exists an isomorphism $\nu \colon X' \isomto X$ of $\order$-modules  and an isomorphism $\tau' \colon X'|_{\mathcal O'} \isomto \mathcal O^n|_{\mathcal O'}$ of $\mathcal O'$-modules. 
Let $\rho'\colon \order\to \End_\Z(X')$ be the structure morphism of the $\order$-module $X'$ and let  $\rho_{X'}\colon \order\to R$ be defined by $\gamma \mapsto \tau' \rho'(\gamma) \tau'^{-1}$ for $\gamma\in \order$.  Consider the automorphism $r= \tau \nu\tau'^{-1}$  of the abelian group $\mathcal O^n$. We see that $r$ is compatible with the diagonal action of the subring $\mathcal O'$ of $\mathcal O$, since $\nu$ is $\order$-compatible and since $\tau,\tau'$ are $\mathcal O'$-compatible. 
Hence $r$ lies in $\End(\mathcal O^n|_{\mathcal O'})^\times =R^\times$. 
Furthermore, on using again that $\nu$ is $\order$-compatible, we compute that  $r\rho_{X'}r^{-1}=\rho_X$. 
This shows that $\rho_X$ and $\rho_{X'}$ coincide in $\Hom_\varphi(\order,R)/R^\times$ and hence we obtain a map 
$$
\iS_{\varphi}\to \Hom_\varphi(\order,R)/R^\times, \ \ [X]\mapsto [\rho_X].$$
To prove that this map is a bijection, we take 
$[X] \in \iS_\varphi$ and we choose an isomorphism
$\tau \colon X|_{\mathcal O'} \isomto \mathcal O^n|_{\mathcal O'}$. Recall that $X_{\rho_X}$ is the abelian group $\mathcal O^n$ equipped with the $\order$-module structure given by $\rho_X\colon \order\to R\subseteq \End_\Z(\mathcal O^n)$. It turns out that the isomorphism $\tau$ is in fact an isomorphism of $\order$-modules $X\isomto X_{\rho_X}$ 
and thus $[X]=[X_{\rho_X}]$. This proves that the composition $[X]\mapsto [\rho_X]\mapsto [X_{\rho_X}]$ is the identity. 
It remains to show that $[X]\mapsto [\rho_X]$ is the inverse of $\rho\mapsto [X_\rho]$. Let $\rho$ be an element of $\Hom_{\varphi}(\order,R)$ and  consider the $\order$-module $X_\rho$ whose underlying abelian group is $\mathcal O^n$ and whose $\order$-module structure is given by $\rho \colon\order\to R\subseteq \End_\Z(\mathcal O^n)$.  We showed that the identity morphism of $\mathcal O^n$ is an isomorphism of  $\mathcal O'$-modules $X_\rho|_{\mathcal O'}\cong \mathcal O^n|_{\mathcal O'}$, and  that   $[\rho_{X_\rho}]$ does not depend on the specific choice of an $\mathcal O'$-isomorphism $\tau:X_\rho|_{\mathcal O'}\cong \mathcal O^n|_{\mathcal O'}$. Therefore, on taking the isomorphism $\tau=\textnormal{id}$, we obtain that $[\rho]=[\rho_{X_\rho}]$ and thus the composition $[\rho]\mapsto [X_\rho]\mapsto [\rho_{X_\rho}]$ is the identity. We conclude that the displayed map is a bijection, proving the proposition. 
\end{proof}

The above result shows that in order to control the number of $\varphi$-compatible $\order$-structures on $R$ it suffices to control $|\iS_\varphi|$. We next embed $\iS_\varphi$ into a certain  monoid $C_\order$  whose cardinality can be bounded via the  theory of orders of number fields.

\subsubsection{Fractional ideals of $\order$}\label{sec:fracideals} 
We continue our notation and terminology. Let $\order$ be an  order of an arbitrary number field $L$. In this section, we study (not necessarily invertible) fractional ideals of $\Gamma$.

We start with some definitions. Let $I$ be a nonzero $\order$-submodule of $L$. We say that $I$ is a fractional ideal of $\order$ if there exists a nonzero $x \in L$ with $x I \subseteq \order$. Then $I$ is a fractional ideal of $\order$ if and only if $I$ is a finitely generated $\Gamma$-submodule of $L$. 
We say that fractional ideals $I,I' $ of $\order$ are equivalent if there exists a nonzero element $x \in L$ such that $I' = xI$. 
We denote by $C_\order$  the set of fractional ideals of $\Gamma$ modulo equivalence. The product of ideals gives $C_\order$ the structure of a commutative monoid with identity element given by the equivalence class of $\order$. 
This equivalence class consists of the fractional ideals $x\order$ with $x\in L$ nonzero. The Picard group $\Pic(\order)$ of $\order$ is the set of invertible elements in $C_\order$.

Let $K$ be a number field which is contained in $L$. For any ring morphism $\varphi:K\to L$, we constructed in \eqref{def:isoclassesofgamma} a set $\iS_\varphi$ consisting of isomorphism classes of certain $\order$-modules. The following result shows that  $\iS_\varphi$ is contained in the monoid $C_\order$.

\begin{lemma}\label{lem:isoembedding}
For any ring morphism $\varphi:K\to L$, there exists an injection $\iS_\varphi \hookrightarrow C_\order$.
\end{lemma}
\begin{proof}
We continue the notation and terminology introduced in the proof of Proposition~\ref{prop:bijection}. To construct a map from $\iS_\varphi$ to $C_\order$, we take $[X] \in \iS_\varphi$. It holds that $\order\otimes_\ZZ\QQ\cong L$. Thus the $\order$-module structure of $X$ gives  $X_\Q = X\otimes_\ZZ \Q$ an $L$-vector space structure and then a $K$-vector space structure $X_\QQ|_K$ via the inclusion $K\subseteq L$. On using that $[X]$ lies in $\iS_\varphi$, we obtain an isomorphism of $\mathcal O'$-modules $X|_{\mathcal O'}\cong \mathcal O^n|_{\mathcal O'}$ which tensored with $\QQ$ becomes an isomorphism of $K$-vector spaces $X_\QQ|_K\cong K^n$. Thus the dimension of the $K$-vector space $X_\QQ|_K$ is $n$. This proves that $X_\QQ$ is a one dimensional $L$-vector space, since the subfield $K\subseteq L$ has relative degree $n=[L:K]$. Hence there exists an isomorphism $\mu \colon X_\Q \isomto L$ of $L$-vector spaces. The image $\mu(X)$ of $X \subset X_\Q$ in $L$ is a nonzero $\order$-submodule of $L$.  On using that $X|_{\mathcal O'}\cong \mathcal O^n|_{\mathcal O'}$, we see that $X$ and thus $\mu(X)$ are finitely generated $\Z$-modules. Therefore $\mu(X)$ is a fractional ideal of $\order$. To see that  $$[X]\mapsto [\mu(X)]$$ defines a map $\iS_\varphi\to C_\order$, we suppose that $X'$ lies in $[X]$.  Then there exists an isomorphism $\nu:X\to X'$ of $\order$-modules and an isomorphism $\mu':X'_\QQ\to L$ of $L$-vector spaces. The map $\mu'(\nu\otimes_\Z\Q)\mu^{-1}:L\to L$ is an automorphism of the one dimensional $L$-vector space $L$, and thus this map is multiplication with a nonzero element $x\in L$. It follows that $\mu'(X')=x\mu(X)$ and hence the fractional ideals $\mu(X)$ and $\mu'(X')$ of $\order$ are equivalent. This shows that $[X]\mapsto [\mu(X)]$ indeed defines a map $\iS_\varphi\to C_\order$.  To see that this map is injective, we assume that the classes $[X']$ and $[X]$ in $\iS_\varphi$  have the same image in $C_\order$. Then there exists  a nonzero $x\in L$ with $\mu(X)=x\mu'(X')$. It follows that $\mu^{-1}(\cdot x)\mu':X'\to X$ is an isomorphism of $\Gamma$-modules and thus  $[X']=[X]$. We conclude that $[X]\mapsto [\mu(X)]$ defines an injection  $\iS_\varphi\hookrightarrow C_\order$. 
This completes the proof of the lemma. 
\end{proof}

To control the cardinality of $C_\order$ we follow a strategy of Hendrik Lenstra; we would like to thank him for explaining us his idea. In a first step we express $C_\order$ as a product of $\Pic(\order)$ with the set $\mathcal I\subseteq C_\order$ of classes which contain a fractional ideal $I$ of $\order$ with $\idf \subseteq I \subseteq \mathcal O_L$, where $\idf$ is the conductor ideal of $\order$ and  $\mathcal O_L$ is the ring of integers of $L$.

\begin{lemma}\label{lem:classnumberbound}
It holds $C_\order=\Pic(\order)\cdot\mathcal I$.
\end{lemma}
\begin{proof}
To simplify notation we write $\mathcal O=\mathcal O_L$ in this proof. On recalling that the monoid $C_\order$ is closed under ideal multiplication, we see that $C_\order$ contains $\Pic(\order)\cdot\mathcal I$. Hence the lemma directly follows from the following claim: Any fractional ideal $I$ of $\order$ takes the form $I=AB$ with fractional ideals $A,B$ of $\order$ such that $[A]$ lies in $\Pic(\order)$ and such that $\idf\subseteq B\subseteq \mathcal O$. 
To prove this claim we take a fractional ideal $I$ of $\order$. The ring $\OL$ is the normalization of $\order$ in $L$,  
 and ideal multiplication with $\OL$ defines a group morphism $\pi:\Pic(\order) \to \Pic(\OL)$.  On using that $\order\subseteq \OL$, we see that  $J = I\OL$  is a fractional ideal of $\OL$ and thus $[J]$ lies in $C_\OL=\Pic(\OL)$. Here $C_\OL$ coincides with $\Pic(\OL)$ since $\OL$ is a Dedekind domain. 
  Further, it follows for example from \cite[p.78]{neukirch:ant} that $\pi$ is surjective. 
Hence there exists  $[A]$ in $\Pic(\order)$  with $A\OL=J^{-1}$. We obtain that $A \OL I = J^{-1} J = \OL,$ which in turn implies $\idf = A\idf I$ since $\idf$ is an ideal of $\mathcal O$.
Thus the inclusion $\idf I\subseteq I$ shows that $\idf$ is contained in $B=AI$. Here   $\idf I$ is contained in $I$, since $I$ is a $\order$-submodule of $L$ and since $\idf$ is contained in $\order$.  Further the equality $A\OL I=\OL$ shows that $B=AI$ is contained in $\OL$. We conclude that $I=A^{-1}B$ where $[A^{-1}]\in \Pic(\order)$ and $\idf\subseteq B\subseteq \OL$. This proves our claim.\end{proof}

Let $\idf'$ be an ideal of $\mathcal O_L$ which is  contained in $\order$, and denote by $\mathcal I'$ the set obtained by replacing $\idf$ with $\idf'$ in the definition of $\mathcal I$. We observe that $\idf'\subseteq\idf$, which gives $\mathcal I\subseteq\mathcal I'$ and hence Lemma~\ref{lem:classnumberbound} implies that $C_\order=\Pic(\order)\cdot \mathcal I'$. However,  to optimize the upper bound for $ |\Pic(\order)|\cdot |\mathcal I'|\geq |C_\order|$, we would like to choose the set $\mathcal I'$ as small as possible and thus we work with  $\mathcal I$. One can explicitly bound  $|\mathcal I|$ and $|\Pic(\order)|$ in terms of the norm $N(\idf)$ of $\idf$, the degree $g=[L:\QQ]$ and the class number $h$ of $\OL_L$. For example it holds
\begin{equation}\label{eq:orderinvarest}
|\mathcal I|\leq N(\idf)^g \ \ \ \textnormal{ and }  \ \ \ |\Pic(\Gamma)|\leq N(\idf)h.
\end{equation}
To see the first inequality, we use that  any group $I$ with $\idf\subseteq I\subseteq \mathcal O_L$ is a free $\ZZ$-module of rank $g$. 
Hence $|\mathcal I|$ is at most the number of  subgroups of $\OL_L/\idf$ which are generated by $g$ elements. This implies the first inequality $|\mathcal I|\leq N(\idf)^g$ in \eqref{eq:orderinvarest}. The second inequality  in \eqref{eq:orderinvarest}  follows from the formula in \cite[p.81]{neukirch:ant} which relates $h$ to the cardinality of $\Pic(\order)$. 
Finally Lemma~\ref{lem:classnumberbound} together with \eqref{eq:orderinvarest}   gives the relatively simple bound  $|C_\order|\leq N(\idf)^{g+1}h$ which is sufficiently strong for the applications in this paper.

\subsubsection{Proof of Proposition~\ref{prop:endoboundmo}}
We continue our notation and terminology. The goal of this section is to prove Proposition~\ref{prop:endoboundmo}. Let $K$, $L$, $\order$ and $R$ be as in the statement of the proposition. 

\begin{proof}[Proof of Proposition~\ref{prop:endoboundmo}]
To bound the number of $\order$-structures on $R$, we use the decomposition of Lemma~\ref{lem:decomp} into $\varphi$-compatible morphisms: The set $\Hom(\order,R)/R^\times$ of $\order$-structures on $R$ takes the form $\cup_{\varphi} \bigl(\Hom_\varphi(\order,R)/R^\times\bigl)$ with the disjoint union taken over all ring morphisms $\varphi:K\to L$.   Proposition~\ref{prop:bijection} gives a bijection $\Hom_\varphi(\order,R)/R^\times\cong\iS_\varphi$ for $\iS_\varphi$ as in \eqref{def:isoclassesofgamma}, and Lemma~\ref{lem:isoembedding} injects $\iS_\varphi$ into the set $C_\order$ of fractional ideals of $\order$ modulo equivalence. Hence, on putting these results together, we obtain an injective map
\begin{equation}\label{eq:endoinjection}
\Hom(\order,R)/R^\times\hookrightarrow \cup_\varphi C_\order
\end{equation}
with the disjoint union taken over all ring morphisms $\varphi:K\to L$.
Lemma~\ref{lem:classnumberbound} gives that
$C_\order=\Pic(\order)\cdot \mathcal I$ for $\mathcal I$ the subset of $C_\order$ defined above Lemma~\ref{lem:classnumberbound}. Then, on combining \eqref{eq:endoinjection} with the bound for  $|\mathcal I|$ given in \eqref{eq:orderinvarest}, we deduce Proposition~\ref{prop:endoboundmo}.
\end{proof}





\subsection{Comparing $\order$-structures on commensurable rings}\label{sec:endoreduction}

In this section we work in a more general setting.  Let  $R$ and $R'$ be not necessarily commutative subrings of an arbitrary $\QQ$-algebra, and let $\order$ be an order of an arbitrary number field. We denote by $d\geq 1$  a rational integer and we consider the subring $\order'=\ZZ[d\order]$ of $\order$ generated by $d\order$. 
The goal of this section is to establish the following:







\begin{proposition}\label{prop:ordercompa}
Suppose that $d R' \subseteq R$ and $d R \subseteq R'$. Then the number of $\Gamma$-structures on $R$ is at most $|R'/dR'|$ times the number of $\order'$-structures on $R'$.
\end{proposition}

We worked out the above result for fairly general rings $R$ and $R'$, since this will be useful for future work and since the arguments do not simplify in the special situations which we encounter in our proof of Theorem~\ref{thm:endobound}. Further, in our applications the ring  $R'$ is a free $\ZZ$-module of finite rank $m$ and for such rings $R'$ it holds that 
\begin{equation}\label{eq:cardofquotient}
|R'/dR'|=d^m.
\end{equation}
Our strategy of proof of Proposition~\ref{prop:ordercompa} is as follows. In Lemma~\ref{lem:propimproved} we first prove a stronger statement in the special case when $R'\subseteq R$. Then we reduce the general case to this special case by working with the ring $R_0=R'\cap R$ and by ``embedding'' via Lemma~\ref{lem:orderunitscompa} the set $R'^\times/R_0^\times$ into $R'/dR'$. Here $R^\times$ denotes the unit group of $R$.  

We start with the following completely elementary observation which describes a useful property of the set $\Hom(\order,R)$ of ring morphisms $\order\to R$.

\begin{lemma}\label{lem:restrictembed}
Suppose that $dR\subseteq R'$. Then the restriction  to the subring $\order'$ of $\order$ defines an injective map from $\Hom(\order,R)$ into $\Hom(\order',R')$. 
\end{lemma}
\begin{proof}
We take $\rho\in \Hom(\order, R)$. It holds that $d \rho(\order) \subseteq d R \subseteq R'$ which proves  $\rho(d\order) \subseteq R'$. It follows that $\rho(\order')$ is contained in $R'$ and thus the restriction $\rho\mapsto\rho|_{\order'}$ defines a map  $\Hom(\order,R)\to\Hom(\order',R')$. To prove that this map is injective, we assume that there exists $\rho'$ in $\Hom(\order,R)$ with $\tau=\rho|_{\order'}=\rho'|_{\order'}$. If $\gamma\in \order$ then  $d\gamma$ lies in $\order'$ and we deduce that $d\rho(\gamma)=\tau(d\gamma)=d\rho'(\gamma)$. It follows that $\rho=\rho'$, since $R$ is contained in a $\QQ$-algebra. We conclude that $\rho\mapsto\rho|_{\order'}$ injects $\Hom(\order,R)$ into $\Hom(\order',R')$ as desired.\end{proof}


We are now ready to prove a stronger version of Proposition~\ref{prop:ordercompa} assuming that $R'\subseteq R$. 

\begin{lemma}\label{lem:propimproved}
Suppose that $d R \subseteq R'$ and assume that  $R' \subseteq R$. Then the number of $\Gamma$-structures on $R$ is at most the number of $\order'$-structures on $R'$.
\end{lemma}
\begin{proof}
Our assumption $R'\subseteq R$ implies that the subgroup $R'^\times$ of $R^\times$ acts on the sets $\Hom(\order, R)$ and $\Hom(\order',R')$ by conjugation. The ring $R'$ contains $d R$ by assumption. Therefore Lemma~\ref{lem:restrictembed} gives that the restriction to  $\order'$ defines an embedding  $\Hom(\order,R)\hookrightarrow \Hom(\order',R')$. This embedding is compatible with the conjugation action of $R'^\times$ on these sets 
and hence we see that the restriction to $\order'$ induces an injective map
\begin{equation*}
\Hom(\order, R)/R'^\times\hookrightarrow \Hom(\order', R')/R'^\times.
\end{equation*} 
On using again that $R'^\times$ is a subgroup of $R^\times$, we obtain a canonical surjection from $\Hom(\order, R)/R'^\times$ onto $\Hom(\order, R)/R^\times$. Thus the number of $\order$-structures on $R$ is at most the cardinality of $\Hom(\order, R)/R'^\times$, which in turn is at most the number of $\order'$-structures on $R'$ in view of the displayed embedding. This implies Lemma~\ref{lem:propimproved}.\end{proof}

In Proposition~\ref{prop:ordercompa} we make the assumption $dR'\subseteq R$. To reduce in a controlled way to the special situation $R'\subseteq R$ of the above lemma, we shall use the following result.

\begin{lemma}\label{lem:orderunitscompa}
Suppose that $dR\subseteq R'$. If $R'\subseteq R$ then $|R^\times/R'^\times|$ is at most $|R/dR|$.
\end{lemma}
\begin{proof}
To prove the statement, we construct an injective correspondence from the set ${R^\times}/{R'^\times}$ to the set $R/d R$. We consider the map $\psi \colon R^\times \to  {R}/{d R}$ obtained by composing the natural projection $R\to R/dR$ with the inclusion $R^\times \subseteq R$. To compute the fibers of $\psi$, we assume that $x, y \in R^\times$ satisfy $\psi(x) = \psi(y)$. Then it holds that $y = x + dr$ for some $r \in R$. We now show that $xy^{-1} \in R'^\times$. For this purpose, we observe that $xy^{-1}=1 - d ry^{-1}$. 
Our assumption $d R \subseteq R'$ assures that $d (ry^{-1})$ lies in $R'$, which proves that $xy^{-1}=1 - d ry^{-1}$ lies in $R'$. Similarly, we see that $yx^{-1}  = 1 + drx^{-1}$ lies in $R'$. It follows that  $z=xy^{-1}$ is a unit of $R'$. In other words, if $x, y \in R^\times$ satisfy $\psi(x) = \psi(y)$ then $x = zy$ for some $z \in R'^\times$. Thus the map $\psi:R^\times \to  {R}/{d R}$ induces an injective correspondence from the set ${R^\times}/{R'^\times}$ to the set $R/d R$, which implies the lemma.  \end{proof}

We remark that the correspondence appearing in the above proof is not necessarily a map, since the construction depends on the choice of a representative of an orbit in $R^\times/R'^\times$. To deduce Lemma~\ref{lem:orderunitscompa}, one can avoid to work with a correspondence.  Indeed the above arguments give that $\psi^{-1}$ defines a surjective map $\psi(R^\times)\to  R^\times/R'^\times$.  We now combine the above results in order to prove Proposition~\ref{prop:ordercompa}.

\begin{proof}[Proof of Proposition~\ref{prop:ordercompa}]
We recall that $R$ and  $R'$ are subrings of a $\QQ$-algebra which we denote by $\Omega$. Our assumptions  $d R \subseteq R'$ and $d R' \subseteq R$ assure that $d R$ and $dR'$ are both contained in the subring $R_0=R\cap R'$ of $\Omega$. 
Furthermore, it holds that $R_0 \subseteq R$. Thus an application of Lemma~\ref{lem:propimproved} with the subrings  $R=R$ and $R'=R_0$ of $\Omega$ gives 
$$
|\Hom(\order, R)/R^\times| \leq |\Hom(\order', R_0)/R_0^\times|.
$$
The identity induces a map from $\Hom(\order',R_0)/R_0^\times$ to $\Hom(\order',R')/R'^\times$, since  $R_0 \subseteq R'$ and $R_0^\times\subseteq R'^\times$. The degree of this map is at most $|R'^\times/R_0^\times|$ 
and then we deduce
$$
|\Hom(\order', R_0)/R_0^\times| \leq |R'^\times/R_0^\times| \cdot |\Hom(\order', R')/R'^\times|.
$$
We already observed  that $dR'\subseteq R_0$, and it holds that $R_0\subseteq R'$. Hence an application of Lemma~\ref{lem:orderunitscompa} with the subrings $R=R'$ and $R'=R_0$ of $\Omega$ gives that $|R'^\times/R_0^\times|\leq |R'/dR'|$. This together with the displayed inequalities implies Proposition~\ref{prop:ordercompa}.
\end{proof}

\subsection{Proof of Theorem~\ref{thm:endobound}}\label{sec:endoproof}

In this section we combine the results established in previous sections in order to prove Theorem~\ref{thm:endobound}. We recall that in Theorem~\ref{thm:endobound} the base scheme $S$ is an open subscheme of  $\spec(\ZZ)$,   and  $\order$ is an order of an arbitrary number field $L$ of degree $g=[L:\QQ]$. Let $A$ be an abelian scheme over $S$ of relative dimension $g$. We now prove an explicit upper bound for the number of $\order$-structures on $A$ by using the strategy outlined in the introduction.

\begin{proof}[Proof of Theorem~\ref{thm:endobound}]
To bound the number of $\order$-structures on $A$, we may and do assume that there exists a ring morphism $\iota:\order\to \End(A)$. The ring $\order\otimes_\ZZ\QQ$ identifies with $L$ and thus tensoring $\iota$ with $\QQ$ gives an embedding $\iota_\QQ:L\hookrightarrow\End^0(A)$. This shows that our abelian scheme $A$ over $S$ of relative dimension $g=[L:\QQ]$ is of $\GL_2$-type. 

1. In a first step we construct an isogenous abelian scheme whose endomorphism ring is of special shape.  Any nonempty open subscheme of $\spec(\ZZ)$ is a connected Dedekind scheme with field of fractions $\QQ$.  Our $A$ over $S$ is of $\GL_2$-type, and $S$ is nonempty. Therefore an application of Lemma~\ref{lem:gl2avstructure}~(iv) with $A$ gives an isogenous abelian scheme $A'$ over $S$ of the form $A'=B^n$, where $B$ is an abelian scheme over $S$ such that $\End(B)$ is a Dedekind ring and such that the generic fiber of $B$ is simple.  In particular, there is an isogeny $\varphi:A\to A'$ with the property that any isogeny $A\to A'$ has degree at least $\deg(\varphi)$. The field $L$ is isomorphic to the subring $L'=\iota_\QQ(L)$ of $\End^0(A)$. Thus an application of Lemma~\ref{lem:gl2avstructure}~(ii) with the number field $L'/\Q$ of degree $g$ provides that $\End^0(B)$ is isomorphic to a subfield $K$ of $L$ such that $L/K$ has  degree $n$ and  $K/\QQ$ has degree $\dim B$. 


2. Next we  compare  $\Gamma$-structures on $A$ with $\order'$-structures on $A'$. Here  $\order'=\ZZ[d\order]$ denotes the subring of $\order$ generated by $d\order$ for $d$  the degree of the isogeny $\varphi:A\to A'$. Let $\varphi^*$ be  the isomorphism of $\QQ$-algebras $\End^0(A)\isomto\End^0(A')$ constructed in \eqref{def:phiupstar}. We  consider the subrings  $R=\varphi^*(\End(A))$ and $R'=\End(A')$ of the $\QQ$-algebra $\End^0(A')$. Lemma~\ref{lem:endorelation}~(i) shows that $dR'\subseteq R$, and inside the $\QQ$-algebra $\End^0(A)$ we obtain the inclusion $d\cdot \End(A)\subseteq \varphi_*(R')$ by Lemma~\ref{lem:endorelation}~(ii).  Here $\varphi_*$ is the isomorphism of $\QQ$-algebras $\End^0(A')\isomto\End^0(A)$ constructed in \eqref{def:philowstar}.  On using that $\varphi_*$ is the inverse of $\varphi^*$, we deduce that $dR\subseteq \varphi^*\varphi_*(R')=R'$. It follows that $R$ and $R'$ are subrings of the $\QQ$-algebra $\End^0(A')$ with $dR\subseteq R'$ and $dR'\subseteq R$. Hence an application of Proposition~\ref{prop:ordercompa} with $R$ and $R'$ gives that the number of $\order$-structures on $R$ is at most $|R'/dR'|$ times the number of $\order'$-structures on $A'$.  To determine the cardinality of $R'/dR'$, we recall that $\End^0(B)\cong K$ and $A'=B^n$. Then we see that the $\QQ$-algebras $\End^{0}(A')$ and $\uM_n(K)$ are isomorphic 
and thus the $\QQ$-dimension of $\End^0(A')$ is $n^2[K:\Q]=ng$. 
Hence the order $R'=\End(A')$  of $\End^0(A')$ is a free $\ZZ$-module of rank $ng$. Therefore \eqref{eq:cardofquotient} gives that $|R'/dR'|=d^{ng}$. Further, the isomorphism $\varphi^*:\End(A)\isomto R$ identifies  $\order$-structures on $A$ with  $\order$-structures on $R$. We conclude that the number of distinct $\order$-structures on $A$ is at most $d^{ng}$ times the number of distinct $\order'$-structures on $A'$.

3. In the next step we control the number of $\order'$-structures on $A'$. We recall that $\End(B)$ is a Dedekind domain and that  $\End^0(B)$ is isomorphic to the number field $K$. It follows that  $\End(B)$ identifies with the ring of integers $\OL$ of $K$ 
and thus  $R'=\End(B^n)$ is isomorphic to $\uM_n(\OL)$. Therefore an application of Proposition~\ref{prop:endoboundmo} with the order $\order'$ of $L$ and the subfield $K\subseteq L$ implies that the number of $\order'$-structures on $A'$ is at most $N(\idf')^{g} h(\order')t.$ Here $N(\idf')$ denotes the norm of the conductor ideal $\idf'$ of $\order'$, and $h(\order')=|\Pic(\order')|$ is the class number of $\order'$. Further  $t$ denotes the number of ring morphisms $K\to L$. To estimate $t$  we consider the tower of fields $\QQ\subseteq K\subseteq L\subseteq L^{cl}$, where  $L^{cl}$ denotes a normal closure of $L/\QQ$. Field theory gives that the set of ring morphisms $K\to L$ injects into the automorphism group $\Aut(L^{cl}/\QQ)$ of the normal extension $L^{cl}/\QQ$. 
Thus $t$ is at most the degree $l=[L^{cl}:\QQ]$. We conclude that the number of $\order'$-structures on $A'$ is at most $N(\idf')^{g} h(\order')l.$

4. In the last step we put everything together and we control the quantities $N(\idf')$, $h(\order')$ and $d$. We begin to estimate $N(\idf')$. Let $\mathcal O_L$ be the ring of integers of $L$, and recall that for any order $\order^*$ of $L$ the  conductor ideal of $\order^*$ consists of the elements $x\in\OL_L$ with $x\OL_L\subseteq \order^*$. This implies that any element $x$ of the conductor ideal $\idf$ of $\order$ satisfies $d x \OL_L \subseteq d \order \subseteq \order'$, which shows that $d \idf\subseteq\idf'$. 
It follows that $N(\idf') \leq N(d\idf)$ and hence we obtain that $N(\idf')\leq d^gN(\idf)$. Further, an application of \eqref{eq:orderinvarest} with the order $\order'$ of $L$ gives that $h(\order')\leq N(\idf')h$ for $h$ the class number of $\mathcal O_L$. Therefore we see that the results proved in steps 2. and 3. imply that the number of $\order$-structures on $A$ is at most
\begin{equation}\label{eq:endoboundwithmindeg}
d^{(2g+1)g}N(\idf)^{g+1}h l. 
\end{equation} 
Here we used in addition that $n=[L:K]$ is at most $[L:\QQ]=g$. It remains to bound the degree $d$. We recall that our abelian scheme $A$ over $S$ is of $\GL_2$-type. Therefore our minimal choice of the isogeny $\varphi:A\to A'$ assures that $d=\deg(\varphi)$ is at most the degree $\deg(\psi)$ of the isogeny $\psi:A\to A'$ which appears in  \eqref{eq:mindegboundshaf}. Thus we see that the upper bound \eqref{eq:mindegboundshaf} for $\deg(\psi)\geq d$  together with \eqref{eq:endoboundwithmindeg} implies Theorem~\ref{thm:endobound}.\end{proof}

\begin{remark}\label{rem:abendo}
Let $\order$ be an order of an arbitrary number field $L$ of degree $g=[L:\QQ]$, and let $A_\QQ$ be an abelian variety over $\QQ$ of dimension $g$. The `minimal' isogeny degree $d(A_\QQ)$ of $A_\QQ$ is the smallest $d\in \ZZ$ with the following property: If $A'_\QQ$ is an abelian variety over $\QQ$ which is isogenous to $A_\QQ$, then there exists an isogeny $A_\QQ\to A'_\QQ$ of degree at most $d$. Then we claim that the number of $\order$-structures on $\End(A_\QQ)$ is at most 
\begin{equation}\label{eq:abendomindeg}
d(A_\QQ)^{(2g+1)g}N(\idf)^{g+1}h l,
\end{equation}
where the quantities $h$, $l$ and $N(\idf)$ are as in Theorem~\ref{thm:endobound}. Indeed this claim follows from \eqref{eq:endoboundwithmindeg} and \eqref{eq:neron}, since $A_\QQ$ extends to an abelian scheme over an open subscheme of $\spec(\ZZ)$. We further mention that the most recent version (due to Gaudron--R\'emond~\cite{gare:isogenies}) of the Masser--W\"ustholz~\cite{mawu:abelianisogenies,mawu:factorization} isogeny estimates provides a fully explicit upper  bound for $d(A_\QQ)$ in terms of $g$ and the stable Faltings height of $A_\QQ$.
\end{remark}

\subsection{A general finiteness result}\label{sec:generalgammastr}

Let $\order$ be an order of a number field. In this section we study more generally $\Gamma$-structures on any abelian scheme $A$ over an arbitrary connected Dedekind scheme $S$. We emphasize that now we do not assume anymore that the relative dimension of $A$ over $S$ is the degree of the number field. We shall see below that the following (non-effective) finiteness result  is a consequence of a special case of the theorem of Jordan--Zassenhaus.

\begin{lemma}\label{lem:generalgammastr}
There are at most finitely many $\Gamma$-structures on $A$.
\end{lemma}

The scheme $S$ is a connected Dedekind scheme and hence \eqref{eq:neron} gives a ring isomorphism $\End(A)\cong \End(A_k)$ for $A_k$ the generic fiber of $A$. Since $A_k$ is an abelian variety, we obtain that  $\End(A)\cong \End(A_k)$ is a finite torsion-free $\ZZ$-algebra and $\End^0(A)\cong \End^0(A_k)$ is a semi-simple $\QQ$-algebra. 
Thus Lemma~\ref{lem:generalgammastr} follows from the next result. 

\begin{lemma}Suppose that $R$ is a finite torsion-free $\ZZ$-algebra such that $R\otimes_\ZZ\QQ$ is a semi-simple $\QQ$-algebra. Then there are at most finitely many $\Gamma$-structures on $R$.\end{lemma}
\begin{proof}
Composing any ring morphism $\rho:\Gamma\to R$ with left multiplication of $R$ on $R$, defines a $\Gamma$ action on $R$ which commutes with the action of $R^{op}$ on $R$ given by right multiplication, where $R^{op}$ is the opposite ring of $R$. Hence any ring morphism $\rho:\Gamma\to R$ defines a $\Lambda$-lattice structure on the free $\ZZ$-module $R$, where $\Lambda=\Gamma\otimes_\ZZ R^{op}$ is an order in the $\QQ$-algebra $\Lambda_\QQ=\Lambda\otimes_\ZZ\QQ$. Here $\Lambda_\QQ\cong L\otimes_\QQ R^{op}_\QQ$ is in fact a finite semi-simple $\QQ$-algebra, since $L=\Gamma\otimes_\ZZ \QQ$ is a number field and since the opposite algebra $R^{op}_\QQ$ of the finite semi-simple $\QQ$-algebra $R_\QQ=R\otimes_\ZZ\QQ$ is again semi-simple and finite.

Suppose that $\rho,\rho':\Gamma\to R$ define isomorphic $\Lambda$-lattice structures on $R$. Then there exists an invertible $\tau\in \End_\ZZ(R)$ such that for all $(\gamma,r)\in \Gamma\times R$ it holds that $(\rho(\gamma)\otimes r)\tau=\tau(\rho'(\gamma)\otimes r)$ inside $\End_\ZZ(R)$. Taking here $\gamma=1$ we see that $\tau$ is compatible with right multiplication of $R$ on $R$ and thus $\tau$ identifies with $\tau(1)\in R^{\times}$. On the other hand, taking here $r=1$ and evaluating at $1$ the resulting morphism in $\End_\ZZ(R)$, we obtain that $\rho=s\rho's^{-1}$ for $s=\tau(1)\in R^\times$; here we used that $\tau(\rho(\gamma)\cdot 1)=\tau(1\cdot \rho(\gamma))=\tau(1)\rho(\gamma)$ since $\tau$ is (right) $R$-compatible. Therefore  $\rho$ and $\rho'$ define the same $\Gamma$-structure on $R$. 

We conclude that  infinitely many distinct $\Gamma$-structures on $R$ would define infinitely many distinct isomorphism classes of $\Lambda$-lattices with underlying $\ZZ$-module isomorphic to $R$. But the latter set is finite by the theorem of Jordan--Zassenhaus, which implies that there can be at most finitely many distinct $\Gamma$-structures on $R$ as claimed. 
\end{proof}

The above finiteness result is not effective since its proof uses the Jordan--Zassenhaus theorem (JZ), see \cite[p.7]{mawu:factorization} for discussions of the effectivity of (JZ). However, in the proof of their factorization estimates, Masser--W\"ustholz were able to make (JZ) effective in the fundamental case when $\End^0(A)$ is a division algebra. In fact factorization estimates were used in our proof of Theorem~\ref{thm:endobound} which gives an explicit uniform version of Lemma~\ref{lem:generalgammastr} for a class of nonsimple abelian schemes $A$ and we expect that such estimates will continue to play an important role in future attempts to make Lemma~\ref{lem:generalgammastr} explicit.

\section{Proof of main results}\label{sec:proofs}

In this section we combine the results obtained in the previous sections to prove our main results stated in Theorem~\ref{thm:main} (i) and (ii). As a by product, we also obtain a proof of Proposition~\ref{prop:northcott} on the Northcott property of the height $h_\phi$.

Let $Y$ be a scheme and let $\mathcal M$ be a Hilbert moduli stack. Suppose that $Y$ is  a Hilbert moduli scheme of some presheaf $\mathcal P$ on $\mathcal M$. Then an application of \eqref{eq:forgetfuldecomp} gives that the natural forgetful map $\phi:Y\to \absg$ of the Hilbert moduli scheme  $Y=M_{\mathcal P}$ decomposes as
$$\phi:Y\to^{\phi_\alpha}M\to^{\phi_\varphi} \abomult\to^{\phi_\iota} M_{\GL_2,g}\hookrightarrow \absg.$$
Here the morphisms $\phi_\alpha$, $\phi_\varphi$, $\phi_\iota$ between the presheaves $Y$, $M$, $\abomult$, $M_{\GL_2,g}$, $\absg$ on $\sch$ are as in Section~\ref{sec:parsin}. 
Let $S$ be a connected Dedekind scheme whose function field is algebraic over $\QQ$. On considering the above decomposition of $\phi$ over $S$, we deduce 
\begin{equation}\label{eq:proofofcardbound}
|Y(S)|\leq \deg(\phi)|M_{\GL_2,g}(S)| \quad \textnormal{and} \quad \deg(\phi)\leq |\mathcal P|_S\deg(\phi_\varphi)\deg(\phi_\iota).
\end{equation} 
Here all morphisms are considered over $S$ and we used the inequality $\deg(\phi_\alpha)\leq |\mathcal P|_S$ obtained in Lemma~\ref{lem:degforgetlvl}. Moreover, on combining the above displayed decomposition of $\phi$ with results proven in previous sections, we obtain the following lemma.
\begin{lemma}\label{lem:finiteiff}
The map $\phi(S)$ is finite if and only if $\mathcal P$ is finite over $\mathcal M(S)$.
\end{lemma}
\begin{proof}
We first assume that the presheaf $\mathcal P$ is finite over $\mathcal M(S)$. By construction via \eqref{quadriso} the fiber of $\phi_\alpha(S):Y(S)\to \hilbmod(S)$ over any $[x]\in \hilbmod(S)$ identifies with the set of isomorphism classes formed by $(x,\alpha)$ with $\alpha\in \mathcal P(x)$, and the set $\mathcal P(x)$ is finite by assumption. This implies that the map $\phi_\alpha(S)$ is finite. We now use that $S$ is a connected Dedekind scheme whose function field is algebraic over $\QQ$. Then the maps $\phi_\varphi(S)$ and $\phi_\iota(S)$ are finite by the proof of Corollary~\ref{cor:forgetpol} and by Lemma~\ref{lem:generalgammastr} respectively.  Thus $\phi=\phi_\iota\phi_\varphi\phi_\alpha$ is a composition of maps which are finite over $S$ and hence $\phi(S)$ is finite.

To prove the converse, we use that $\mathcal M$ is a Deligne--Mumford stack and that $S$ is quasi-compact. This assures that the automorphism group of any object in $\mathcal M(S)$ is finite. 
We now assume that $\mathcal P$ is not finite over $\mathcal M(S)$. Then there exists $x\in\mathcal M(S)$ with $\mathcal P(x)$ infinite. The fiber of $\phi_\alpha(S)$ over $[x]\in\hilbmod(S)$ identifies with the set of classes formed by the infinitely many pairs $(x,\alpha)$ with $\alpha$ in the infinite set $\mathcal P(x)$, and these infinitely many pairs form infinitely many classes since the automorphism group of $x$ is finite. This implies that $\phi_\alpha$ and $\phi=\phi_\iota\phi_\varphi\phi_\alpha$ are not finite over $S$. We conclude the converse: If $\phi(S)$ is finite, then $\mathcal P$ is finite over $\mathcal M(S)$. This completes the proof of the lemma.  
\end{proof}

The above lemma is crucial for our proof of Proposition~\ref{prop:northcott} on the Northcott property, but it will not be used in the following proof of  Theorem~\ref{thm:main}.

\subsection{Proof of Theorem~\ref{thm:main}}
We  continue our notation. As in Theorem~\ref{thm:main}, we assume that $S\subseteq \spec(\ZZ)$ is nonempty open, that $Y$ is a variety over $\ZZ$ and  that there is a nonempty open $T\subseteq \spec(\ZZ)$ such that $Y_T=M_{\mathcal P}$ is a Hilbert moduli scheme of some presheaf $\mathcal P$ on $\mathcal M$. 

\begin{proof}[Proof of Theorem~\ref{thm:main}]
We first show that $Y(S)\subseteq Y_T(U)$ where $U=T\cap S$. The schemes $U$ and $S$ are nonempty open subschemes of the irreducible scheme $B=\spec(\ZZ)$. Thus $U\subseteq S$ is a dense open subscheme, which implies that $Y(S)\subseteq Y(U)$ since $Y$ is separated over $B$. Further it holds that $Y(U)\cong Y_T(U)$, since $U\subseteq T$ are both open subschemes of $B$ and $B$ is the terminal object in $\sch$. It follows that $Y(S)\subseteq Y_T(U)$ as desired. 

We now prove (i). On recalling the construction of the height $h_\phi$ on $Y(S)$ appearing in (i), we see that this height is the restriction to the subset $Y(S)\subseteq Y_T(U)$ of the height $h_\phi$ on $Y_T(U)$ defined in \eqref{def:height}. The latter height is the pullback of the stable Faltings height $h_F$ along the forgetful map $\phi:Y_T\to \absg$  of the Hilbert moduli scheme $Y_T=M_{\mathcal P}$, and the forgetful map $\phi$ factors through $M_{\GL_2,g}$ by \eqref{eq:forgetfuldecomp}. Therefore any point $P$ in $Y(S)\subseteq Y_T(U)$ satisfies $h_\phi(P)\leq \sup h_F(A)$ with the supremum taken over all $A\in M_{\GL_2,g}(U)$. Then, on applying the effective Shafarevich conjecture in \eqref{eq:esgl2} with the nonempty open subscheme $U\subseteq \spec(\ZZ)$, we obtain for all $A\in M_{\GL_2,g}(U)$ an upper bound for $h_F(A)$ in terms of $N_U$ and $g$ which together with $h_\phi(P)\leq \sup h_F(A)$ directly proves  (i).

To show the upper bound for $|Y(S)|$ claimed in (ii), we use the inclusion $Y(S)\subseteq Y_T(U)$ and we bound $|Y_T(U)|$. The scheme $U$ is a connected Dedekind scheme with function field $\QQ$. Therefore Corollary~\ref{cor:forgetpol} gives that $\deg(\phi_\varphi)\leq 2^g$ over $U$ and then an application of \eqref{eq:proofofcardbound} with the forgetful map $\phi$ of the Hilbert moduli scheme $Y_T=M_{\mathcal P}$ implies  $$|Y_T(U)|\leq \deg(\phi)|M_{\GL_2,g}(U)| \quad \textnormal{and} \quad \deg(\phi)\leq 2^g|\mathcal P|_U\deg(\phi_\iota)$$ 
over $U$; here $\deg(\phi_\iota)=1$ when $g=1$. To obtain bounds for $|M_{\GL_2,g}(U)| $ and $\deg(\phi_\iota)$ over $U$, we apply with $S=U$ the quantitative Shafarevich conjecture \eqref{eq:qesgl2} and Corollary~\ref{cor:forgetendo} respectively. These bounds together with the above displayed inequalities lead to 
\begin{equation}\label{eq:mainthmii}
|Y(S)|\leq |Y_T(U)|\leq |\textnormal{Pic}(\OL)||\mathcal P|_U (4g)^{(8g)^7}N_U^{(12g)^5},
\end{equation}
where $\OL$ denotes the ring of integers appearing in the definition of $\mathcal M$. This implies the bound for $|Y(S)|$ claimed in (ii) and hence completes the proof of Theorem~\ref{thm:main}. 
\end{proof}

We remark that one can improve up to a certain extent the exponents in Theorem~\ref{thm:main}~(i) and \eqref{eq:mainthmii} without introducing substantial new ideas. For example, one can go into the proofs of \cite{rvk:modularhms} in which the explicit inequalities \eqref{eq:esgl2}, \eqref{eq:qesgl2} and \eqref{eq:mindegboundshaf} were shown.

On the other hand, substantial new ideas are required to generalize to arbitrary number fields $K$ our explicit Theorem~\ref{thm:main}. For instance, in our proofs over $K=\QQ$ we crucially exploit a geometric version of modularity based inter alia on the Tate conjecture, while for an arbitrary $K$ such a geometric version of modularity is not available in general. However, Diophantine applications of modularity are certainly not restricted to $K=\QQ$: On using  modularity results for elliptic curves, 
several authors obtained Diophantine applications over certain $K\neq \QQ$; see for example Jarvis--Meekin~\cite{jame:fermat}, Freitas--Siksek~\cite{frsi:fermat56}, Pasten~\cite{pasten:shimabc} and Freitas--Kraus--Siksek~\cite{frkrsi2:asymfermat} and the references therein.

\subsection{Proof of Proposition~\ref{prop:northcott}}
We continue our notation. Let $K$ be a number field field with ring of integers $\mathcal O_K$. As in Proposition~\ref{prop:northcott}, we assume that $Y$ is a variety over $\mathcal O_K$ and we suppose that $Y$ is a Hilbert moduli scheme of some presheaf $\mathcal P$ on $\mathcal M$ such that $\mathcal P$ is finite over $\mathcal M(\bar{K})$. Let $\phi=\phi_\mathcal P$ be the forgetful map of $Y=M_{\mathcal P}$ defined in \eqref{def:forgetfulmap} and write $f=\phi(\bar{K})$.

\begin{proof}[Proof of Proposition~\ref{prop:northcott}]The presheaf $\mathcal P$ is finite over $\mathcal M(\bar{K})$ by assumption, and $S=\spec(\bar{K})$ is a connected Dedekind scheme whose function field $\bar{K}$ is algebraic over $\QQ$. Therefore Lemma~\ref{lem:finiteiff}  gives that the forgetful map $f:Y(S)\to \absg(S)$ has finite fibers. Let $d\in\ZZ$ be positive and let $c>0$ be a real number. Denote by $\mathcal B$ the set of points $P\in Y(\bar{K})$ such that $[K(P):K]\leq d$ and such that $h_\phi(P)\leq c$. The forgetful map $\phi:Y\to \absg$ is a morphism of presheaves on $\sch$. Hence, for any $P\in Y(\bar{K})$, the class $f(P)$ contains an abelian variety over $\bar{K}$ which is the base change of an abelian variety over $K(P)$. Thus  the classes in $f(\mathcal B)$ are generated by abelian varieties $A$ over $\bar{K}$ of dimension $g$ such that $h_F(A)\leq c$ and such that $A$ is the base change of an abelian variety defined over the number field $K(P)$ which has degree $[K(P):\QQ]\leq d[K:\QQ]$. Hence the Northcott property \cite[p.169]{fach:deg} of $h_F$ on $\absg(\bar{\QQ})\cong \absg(\bar{K})$ implies that $f(\mathcal B)$ is finite. This shows that $f^{-1}(f(\mathcal B))$ is finite since $f$ has finite fibers, and then $\mathcal B\subseteq f^{-1}(f(\mathcal B))$ is finite. We conclude that $h_\phi$ has the Northcott property. This completes the proof of Proposition~\ref{prop:northcott}.   
\end{proof}

\subsection{Proof of Corollary~\ref{cor:p1n}}\label{sec:proofp1n}

We continue our notation. In this section we deduce Corollary~\ref{cor:p1n} in which we consider the presheaf $\mathcal P_1(\mathfrak n)$ on $\mathcal M$ defined in \eqref{def:p1n} for any ideal $\mathfrak n\subseteq \OL$ of norm $n\geq 1$.

\begin{proof}[Proof of Corollary~\ref{cor:p1n}]
Write $\mathcal P=\mathcal P_1(\mathfrak n)$ and  $T=\spec(\ZZ[1/n])$. By assumption the restriction of the presheaf $\mathcal P$ on $\mathcal M$ to the open substack $\mathcal M_{T}$ is representable by an object $y$ of $\mathcal M_{T}(Y_T)$. In particular $Y_T=M_{\mathcal P'}$ is a Hilbert moduli scheme of the (naive extension) presheaf $\mathcal P'$ on $\mathcal M$ represented by the object $y$ of $\mathcal M$.  The two presheaves $\mathcal P'$ and $\mathcal P$ on $\mathcal M$ coincide over $\mathcal M_{T}$, and the nonempty open subscheme $U=T\cap S$ of $\spec(\ZZ)$ satisfies $N_U\leq \prod p$ with the product taken over all rational primes $p\mid nN_S$. Thus an application of Theorem~\ref{thm:main} with $Y$, $T$ and $\mathcal P'$ directly implies the height bound stated in (i). Further, Lemma~\ref{lem:geomlvlred} gives that $|\mathcal P'|_U\leq |\mathcal P'|_\CC=|\mathcal P|_\CC$, and it follows from  \eqref{def:p1n} that $|\mathcal P|_\CC\leq n^{2g}$. Then the slightly more precise version of Theorem~\ref{thm:main}~(ii) obtained in \eqref{eq:mainthmii} leads to the upper bound for the number of $S$-points stated in (ii). This completes the proof. 
\end{proof}

\section{A finiteness result for arbitrary number fields}\label{sec:generalfin}

Let $K$ be a number field field with ring of integers $\mathcal O_K$, let  $Y$ be a variety over $B=\spec(\mathcal O_K)$ and let $S\subseteq B$ be nonempty open. As discussed, substantial new ideas are required to generalize to a general $K$ the explicit results for $Y(S)$ obtained in Theorem~\ref{thm:main} when $K=\QQ$. However, in this section we show how to modify the proof of Theorem~\ref{thm:main} in order to obtain the following finitess result for $Y(S)$ when $K$ is a general number field:  

\begin{proposition}\label{prop:fin}
The set $Y(S)$ is finite if there is a nonempty open $T\subseteq B$ such that $Y_T$ is a Hilbert moduli scheme of a presheaf $\mathcal P$ on a Hilbert moduli stack  with $|\mathcal P|_{S\cap T}<\infty$. 
\end{proposition}
We shall see below that many interesting cases of this finiteness result are already in the literature; although the statements are formulated in a different way and the translation usually requires some work. The presheaf $\mathcal P$ in Proposition~\ref{prop:fin} satisfies\footnote{This follows from Lemma~\ref{lem:geomlvlred}, since $S\cap T$ is nonempty open in the integral scheme $T$ with function field $k(T)=K$ and since $Y_T=M_{\mathcal P}$ is a variety (over the affine scheme $T$ which is of finite type) over $\ZZ$.} $|\mathcal P|_{S\cap T}\leq |\mathcal P|_{F}$ for any field $F$ such that $K$ embeds into $F$, and for those moduli problems  $\mathcal P$ of interest in arithmetic it is usually possible  to show that $|\mathcal P|_{\CC}<\infty$ via  geometric arguments. 

To obtain Proposition~\ref{prop:fin}, it suffices to modify our proof of Theorem~\ref{thm:main} as follows: We can not anymore apply the quantitative Shafarevich conjecture for $M_{\GL_2,g}$ in \eqref{eq:qesgl2} and the explicit Theorem~\ref{thm:endobound}, since they both assume that $K=\QQ$. Instead we use the Shafarevich conjecture for unpolarized abelian varieties over any $K$ proven by Faltings--Zarhin~\cite{faltings:finiteness,zarhin:shafarevich} and we apply Lemma~\ref{lem:generalgammastr} based on the Jordan--Zassenhaus theorem.

\begin{proof}[Proof of Proposition~\ref{prop:fin}]
First, we observe that $U=S\cap T$ is a connected Dedekind scheme whose function field is $K$, since $U$ is a nonempty open subscheme of $B=\spec(\OL_K)$. Our assumption $|\mathcal P|_{U}<\infty$ implies that $\mathcal P$ is finite over $\mathcal M(U)$ for $\mathcal M$ the involved Hilbert moduli stack. Then the natural forgetful map $\phi:Y_T=M_\mathcal P\to \absg$ is finite over $U$ by  Lemma~\ref{lem:finiteiff} which is based on Lemma~\ref{lem:generalgammastr}. Moreover, an application of the unpolarized Shafarevich conjecture \cite[Thm 1]{zarhin:shafarevich} with the open $U\subseteq B$ gives that $\absg(U)$ is finite. Thus the set $\Hom_{\sch}(U,Y_T)$ is finite, since it is the preimage of $\absg(U)$ under the finite map $\phi(U)$.  Next, on using that $U\subseteq S$ are both nonempty open in the irreducible $B$ and that $Y$ is separated over $B$, we obtain that $Y(S)\subseteq Y(U)$. Then $Y(S)$ is finite, since $U\subseteq T$ and thus $Y(U)=\Hom_{B}(U,Y)$ is contained in the finite set $\Hom_{\sch}(U,Y_T)$.
\end{proof}

We now briefly discuss related finiteness results. For each $Y$ and $S$ as above, finiteness of $Y(S)$  follows directly from the polarized Shafarevich conjecture \cite[Satz 6]{faltings:finiteness} when there exist a number field $K'$, an open $S'\subseteq\spec(\mathcal O_{K'})$ and a finite map of sets
$$
Y(S)\to A_{g,d}(S').
$$
For example, Deligne--Szpiro and Ullmo used \'etale covers and other tools in order to construct such finite maps $Y(S)\to A_{g,d}(S')$ in the following two cases: $Y_\CC$ is a compact quotient of the Siegel space by a discrete group without torsion (see \cite[Rem 5.3]{szpiro:faltings}) and $Y_\CC=\Gamma\backslash X^+$ is an adjoint (connected) Shimura variety $(X,G)$ of abelian type with $\Gamma\subset G(\QQ)$ neat (see\footnote{The precise statement requires additional assumptions on the integral model $Y$.} \cite[Thm 3.2]{ullmo:ratpoints}). In particular, in these cases Deligne--Szpiro and Ullmo obtained finiteness of $Y(S)$ and then finiteness of  $Y(K)=Y(S)$ when $Y$ is in addition projective over $B$. These finiteness results of Deligne--Szpiro and Ullmo are quite remarkable, since for higher dimensional varieties it is very difficult to obtain finiteness results for $Y(S)$ or even $Y(K)$; see for example Corvaja--Zannier~\cite{coza:intpointssurfaces} and Levin~\cite{levin:siegelpicard} for some results using a very different method based on Diophantine approximations.

We point out that one can use the fairly general result \cite[Thm 3.2]{ullmo:ratpoints} to deduce Proposition~\ref{prop:fin} for many $Y$ of interest. However, there are also many interesting $Y$ for which the finiteness result of $Y(S)$ obtained in Proposition~\ref{prop:fin} is new; note that neatness implies representability of the moduli problem but the converse is false in general. Finally, we mention that in general, that is for each $Y$ and $S$ as in Proposition~\ref{prop:fin}, it is not clear to us how to obtain a map $f:Y(S)\to A_{g,d}(S')$ for which one can directly verify the finiteness of $f$. In light of this, we avoided the approach (used by Deligne--Szpiro and Ullmo) involving $A_{g,d}$ and instead we proved Proposition~\ref{prop:fin} by combining the strategy of Theorem~\ref{thm:main} with the unpolarized Shafarevich conjecture proven by Faltings and Zarhin.

\newpage

{\scriptsize
\bibliographystyle{amsalpha}
\bibliography{../../../literature}

\def\cprime{$'$}
\providecommand{\bysame}{\leavevmode\hbox to3em{\hrulefill}\thinspace}
\providecommand{\MR}{\relax\ifhmode\unskip\space\fi MR }
\providecommand{\MRhref}[2]{%
  \href{http://www.ams.org/mathscinet-getitem?mr=#1}{#2}
}
\providecommand{\href}[2]{#2}
\begin{thebibliography}{HVdV74}

\bibitem[Ana73]{anantharaman:groupes}
S.~Anantharaman, \emph{Sch\'{e}mas en groupes, espaces homog\`enes et espaces
  alg\'{e}briques sur une base de dimension 1}, 5--79. Bull. Soc. Math. France,
  M\'{e}m. 33.

\bibitem[Baa10]{baaziz:x1nmodels}
H.~Baaziz, \emph{Equations for the modular curve {$X_1(N)$} and models of
  elliptic curves with torsion points}, Math. Comp. \textbf{79} (2010),
  no.~272, 2371--2386.

\bibitem[Bil95]{bilu:israel}
Y.~Bilu, \emph{Effective analysis of integral points on algebraic curves},
  Israel J. Math. \textbf{90} (1995), no.~1-3, 235--252.

\bibitem[Bil02]{bilu:modular}
Y.~F. Bilu, \emph{Baker's method and modular curves}, A panorama of number
  theory or the view from {B}aker's garden ({Z}\"urich, 1999), Cambridge Univ.
  Press, Cambridge, 2002, pp.~73--88.

\bibitem[BLR90]{bolura:neronmodels}
S.~Bosch, W.~L{\"u}tkebohmert, and M.~Raynaud, \emph{N\'eron models},
  Ergebnisse der Mathematik und ihrer Grenzgebiete (3) [Results in Mathematics
  and Related Areas (3)], vol.~21, Springer-Verlag, Berlin, 1990.

\bibitem[Bos96]{bost:lowerbound}
J.-B. Bost, \emph{Arakelov geometry of abelian varieties}, Conference on
  {A}rithmetical {G}eometry, Max {P}lanck {I}nstitut f\"ur {M}athematik {B}onn,
  vol. 96-51, 1996, pp.~1--6.

\bibitem[CCO14]{chcooo:cmbook}
C.-L. Chai, B.~Conrad, and F.~Oort, \emph{Complex multiplication and lifting
  problems}, Mathematical Surveys and Monographs, vol. 195, American
  Mathematical Society, Providence, RI, 2014.

\bibitem[Con04]{conrad:gz}
B.~Conrad, \emph{Gross-{Z}agier revisited}, Heegner points and {R}ankin
  {$L$}-series, Math. Sci. Res. Inst. Publ., vol.~49, Cambridge Univ. Press,
  Cambridge, 2004, With an appendix by W. R. Mann, pp.~67--163.

\bibitem[CZ04]{coza:intpointssurfaces}
P.~Corvaja and U.~Zannier, \emph{On integral points on surfaces}, Ann. of Math.
  (2) \textbf{160} (2004), no.~2, 705--726.

\bibitem[DM69]{demu:stacks}
P.~Deligne and D.~Mumford, \emph{The irreducibility of the space of curves of
  given genus}, Inst. Hautes \'{E}tudes Sci. Publ. Math. (1969), no.~36,
  75--109.

\bibitem[DP94]{depa:hilbertmodular}
P.~Deligne and G.~Pappas, \emph{Singularit\'es des espaces de modules de
  {H}ilbert, en les caract\'eristiques divisant le discriminant}, Compositio
  Math. \textbf{90} (1994), no.~1, 59--79.

\bibitem[DR73]{dera:modules}
P.~Deligne and M.~Rapoport, \emph{Les sch\'emas de modules de courbes
  elliptiques}, Modular functions of one variable, {II} ({P}roc. {I}nternat.
  {S}ummer {S}chool, {U}niv. {A}ntwerp, {A}ntwerp, 1972), Springer, Berlin,
  1973, pp.~143--316. Lecture Notes in Math., Vol. 349.

\bibitem[Fal83]{faltings:finiteness}
G.~Faltings, \emph{Endlichkeitss\"atze f\"ur abelsche {V}ariet\"aten \"uber
  {Z}ahlk\"orpern}, Invent. Math. \textbf{73} (1983), no.~3, 349--366.

\bibitem[FC90]{fach:deg}
G.~Faltings and C.-L. Chai, \emph{Degeneration of abelian varieties},
  Ergebnisse der Mathematik und ihrer Grenzgebiete (3) [Results in Mathematics
  and Related Areas (3)], vol.~22, Springer-Verlag, Berlin, 1990, With an
  appendix by David Mumford.

\bibitem[FKS19]{frkrsi2:asymfermat}
N.~Freitas, A.~Kraus, and S.~Siksek, \emph{Class field theory, {D}iophantine
  analysis and the asymptotic {F}ermat's {L}ast {T}heorem}, Preprint,
  arXiv:1902.07798 (2019), 29 pages.

\bibitem[Fre97]{frey:ternary}
G.~Frey, \emph{On ternary equations of {F}ermat type and relations with
  elliptic curves}, Modular forms and {F}ermat's last theorem ({B}oston, {MA},
  1995), Springer, New York, 1997, pp.~527--548.

\bibitem[FS15]{frsi:fermat56}
N.~Freitas and S.~Siksek, \emph{The asymptotic {F}ermat's last theorem for
  five-sixths of real quadratic fields}, Compos. Math. \textbf{151} (2015),
  no.~8, 1395--1415.

\bibitem[GR14]{gare:isogenies}
\'{E}. Gaudron and G.~R\'{e}mond, \emph{Polarisations et isog\'{e}nies}, Duke
  Math. J. \textbf{163} (2014), no.~11, 2057--2108.


\bibitem[vdG88]{vandergeer:hilbertmodular}
G.~van~der Geer, \emph{Hilbert modular surfaces}, Ergebnisse der Mathematik und
  ihrer Grenzgebiete (3) [Results in Mathematics and Related Areas (3)],
  vol.~16, Springer-Verlag, Berlin, 1988.

\bibitem[HVdV74]{hiva:hmsclass}
F.~Hirzebruch and A.~Van~de Ven, \emph{Hilbert modular surfaces and the
  classification of algebraic surfaces}, Invent. Math. \textbf{23} (1974),
  1--29.

\bibitem[HZ77]{hiza:hmsclass}
F.~Hirzebruch and D.~Zagier, \emph{Classification of {H}ilbert modular
  surfaces}, 43--77.

\bibitem[Jav14]{javanpeykar:belyi}
A.~Javanpeykar, \emph{Polynomial bounds for {A}rakelov invariants of {B}elyi
  curves}, Algebra Number Theory \textbf{8} (2014), no.~1, 89--140, With an
  appendix by Peter Bruin.

\bibitem[Jin15]{jin:torsion}
J.~Jin, \emph{Homogeneous division polynomials for {W}eierstrass elliptic
  curves}, Preprint, arXiv:1303.4327v3 (2015), 15 pages.

\bibitem[JM04]{jame:fermat}
F.~Jarvis and P.~Meekin, \emph{The {F}ermat equation over {${\Bbb
  Q}(\sqrt{2})$}}, J. Number Theory \textbf{109} (2004), no.~1, 182--196.

\bibitem[vK13]{rvk:modularhms}
R.~von K\"anel, \emph{Modularity and integral points on moduli schemes},
  arXiv:1310.7263 (2013), 75 pages, Sections 1-7 are in
  \cite{rvk:intpointsmodellhms}.

\bibitem[vK14]{rvk:intpointsmodellhms}
\bysame, \emph{Integral points on moduli schemes of elliptic curves}, Trans.
  London Math. Soc. \textbf{1} (2014), no.~1, 85--115, Sections 1-7 of
  \cite{rvk:modularhms}.

\bibitem[vKK]{vkkr:repcond}
R.~von K\"anel and A.~Kret, \emph{On the representability of moduli problems on
  {H}ilbert moduli stacks}, In preparation.

\bibitem[vKM16]{vkma:computation}
R.~von K\"anel and B.~Matschke, \emph{Solving {$S$}-unit, {M}ordell, {T}hue,
  {T}hue--{M}ahler and generalized {R}amanujan--{N}agell equations via
  {S}himura-{T}aniyama conjecture}, Preprint, arXiv:1605.06079 (2016), 161
  pages.

\bibitem[KM85]{kama:moduli}
N.~M. Katz and B.~Mazur, \emph{Arithmetic moduli of elliptic curves}, Annals of
  Mathematics Studies, vol. 108, Princeton University Press, Princeton, NJ,
  1985.

\bibitem[Kot92]{kottwitz:shimvar}
R.~Kottwitz, \emph{Points on some {S}himura varieties over finite fields}, J.
  Amer. Math. Soc. \textbf{5} (1992), no.~2, 373--444.

\bibitem[KW09]{khwi:serre}
C.~Khare and J.-P. Wintenberger, \emph{Serre's modularity conjecture}, Invent.
  Math. \textbf{178} (2009), no.~3, 485--586.

\bibitem[Lan83]{lang:cm}
S.~Lang, \emph{Complex multiplication}, Grundlehren der Mathematischen
  Wissenschaften [Fundamental Principles of Mathematical Sciences], vol. 255,
  Springer-Verlag, New York, 1983.

\bibitem[Lev08]{levin:intpointsrunge}
A.~Levin, \emph{Variations on a theme of {R}unge: effective determination of
  integral points on certain varieties}, J. Th\'{e}or. Nombres Bordeaux
  \textbf{20} (2008), no.~2, 385--417.

\bibitem[Lev09]{levin:siegelpicard}
\bysame, \emph{Generalizations of {S}iegel's and {P}icard's theorems}, Ann. of
  Math. (2) \textbf{170} (2009), no.~2, 609--655.

\bibitem[Lev14]{levin:intpointslogforms}
\bysame, \emph{Linear forms in logarithms and integral points on
  higher-dimensional varieties}, Algebra Number Theory \textbf{8} (2014),
  no.~3, 647--687.

\bibitem[LF17]{lefourn:A2rungeparis}
S.~Le~Fourn, \emph{Sur la m\'{e}thode de {R}unge et les points entiers de
  certaines vari\'{e}t\'{e}s modulaires de {S}iegel}, C. R. Math. Acad. Sci.
  Paris \textbf{355} (2017), no.~8, 847--852.

\bibitem[LF19a]{lefourn:tubularbaker}
\bysame, \emph{Tubular approaches to {B}aker's method for curves and
  varieties}, Preprint: arXiv:1812.06306v2 (2019), 11 pages.

\bibitem[LF19b]{lefourn:A2runge}
\bysame, \emph{A tubular variant of {R}unge's method in all dimensions, with
  applications to integral points on {S}iegel modular varieties}, Algebra
  Number Theory \textbf{13} (2019), no.~1, 159--209.

\bibitem[LMB00]{lamo:stacks}
G.~Laumon and L.~Moret-Bailly, \emph{Champs alg\'{e}briques}, Ergebnisse der
  Mathematik und ihrer Grenzgebiete. 3. Folge. A Series of Modern Surveys in
  Mathematics [Results in Mathematics and Related Areas. 3rd Series. A Series
  of Modern Surveys in Mathematics], vol.~39, Springer-Verlag, Berlin, 2000.

\bibitem[Maz77]{mazur:eisenstein}
B.~Mazur, \emph{Modular curves and the {E}isenstein ideal}, Inst. Hautes
  \'{E}tudes Sci. Publ. Math. (1977), no.~47, 33--186 (1978).

\bibitem[MP13]{mupa:modular}
M.~R. Murty and H.~Pasten, \emph{Modular forms and effective {D}iophantine
  approximation}, J. Number Theory \textbf{133} (2013), no.~11, 3739--3754.

\bibitem[Mum65]{mumford:picardmoduli}
D.~Mumford, \emph{Picard groups of moduli problems}, Arithmetical {A}lgebraic
  {G}eometry ({P}roc. {C}onf. {P}urdue {U}niv., 1963), Harper \& Row, New York,
  1965, pp.~33--81.

\bibitem[Mum08]{mumford:av}
\bysame, \emph{Abelian varieties}, Tata Institute of Fundamental Research
  Studies in Mathematics, vol.~5, Published for the Tata Institute of
  Fundamental Research, Bombay; by Hindustan Book Agency, New Delhi, 2008, With
  appendices by C. P. Ramanujam and Yuri Manin, Corrected reprint of the second
  (1974) edition.

\bibitem[MW93]{mawu:abelianisogenies}
D.~W. Masser and G.~W{\"u}stholz, \emph{Isogeny estimates for abelian
  varieties, and finiteness theorems}, Ann. of Math. (2) \textbf{137} (1993),
  no.~3, 459--472.

\bibitem[MW95]{mawu:factorization}
\bysame, \emph{Factorization estimates for abelian varieties}, Inst. Hautes
  \'Etudes Sci. Publ. Math. (1995), no.~81, 5--24.

\bibitem[Neu99]{neukirch:ant}
J.~Neukirch, \emph{Algebraic number theory}, Grundlehren der Mathematischen
  Wissenschaften [Fundamental Principles of Mathematical Sciences], vol. 322,
  Springer-Verlag, Berlin, 1999, Translated from the 1992 German original and
  with a note by Norbert Schappacher, With a foreword by G. Harder.

\bibitem[NN81]{nano:polarizations}
M.~S. Narasimhan and M.~V. Nori, \emph{Polarisations on an abelian variety},
  Proc. Indian Acad. Sci. Math. Sci. \textbf{90} (1981), no.~2, 125--128.

\bibitem[Pap95]{pappas:hilbmod}
G.~Pappas, \emph{Arithmetic models for {H}ilbert modular varieties}, Compositio
  Math. \textbf{98} (1995), no.~1, 43--76.

\bibitem[Par68]{parshin:construction}
A.~N. Par{\v{s}}in, \emph{Algebraic curves over function fields. {I}}, Izv.
  Akad. Nauk SSSR Ser. Mat. \textbf{32} (1968), 1191--1219.

\bibitem[Pas17]{pasten:shimabc}
H.~Pasten, \emph{Shimura curves and the abc conjecture}, Preprint,
  arXiv:1705.09251v4 (2017), 88 pages.

\bibitem[Rap78]{rapoport:hilbertmodular}
M.~Rapoport, \emph{Compactifications de l'espace de modules de
  {H}ilbert-{B}lumenthal}, Compositio Math. \textbf{36} (1978), no.~3,
  255--335.

\bibitem[Ray85]{raynaud:abelianisogenies}
M.~Raynaud, \emph{Hauteurs et isog\'enies}, Ast\'erisque (1985), no.~127,
  199--234, Seminar on arithmetic bundles: the Mordell conjecture (Paris,
  1983/84).

\bibitem[Rei75]{reiner:orders}
I.~Reiner, \emph{Maximal orders}, Academic Press [A subsidiary of Harcourt
  Brace Jovanovich, Publishers], London-New York, 1975, London Mathematical
  Society Monographs, No. 5.

\bibitem[R{\'e}m99]{remond:construction}
G.~R{\'e}mond, \emph{Hauteurs th\^eta et construction de {K}odaira}, J. Number
  Theory \textbf{78} (1999), no.~2, 287--311.

\bibitem[Rib92]{ribet:gl2}
K.~A. Ribet, \emph{Abelian varieties over {${\bf Q}$} and modular forms},
  Algebra and topology 1992 ({T}aej\u on), Korea Adv. Inst. Sci. Tech., Taej\u
  on, 1992, pp.~53--79.

\bibitem[Sha14a]{sha:intmod2}
M.~Sha, \emph{Bounding the {$j$}-invariant of integral points on certain
  modular curves}, Int. J. Number Theory \textbf{10} (2014), no.~6, 1545--1551.

\bibitem[Sha14b]{sha:intmodimrn}
\bysame, \emph{Bounding the {$j$}-invariant of integral points on modular
  curves}, Int. Math. Res. Not. IMRN (2014), no.~16, 4492--4520.

\bibitem[Sut12]{sutherland:torsion}
A.~Sutherland, \emph{Constructing elliptic curves over finite fields with
  prescribed torsion}, Math. Comp. \textbf{81} (2012), no.~278, 1131--1147.

\bibitem[Szp85]{szpiro:faltings}
L.~Szpiro, \emph{La conjecture de {M}ordell (d'apr\`es {G}. {F}altings)},
  Ast\'erisque (1985), no.~121-122, 83--103, Seminar Bourbaki, Vol. 1983/84.

\bibitem[Ull04]{ullmo:ratpoints}
E.~Ullmo, \emph{Points rationnels des vari\'{e}t\'{e}s de {S}himura}, Int.
  Math. Res. Not. (2004), no.~76, 4109--4125.

\bibitem[Vol05]{vollaard:hilbmod}
I.~Vollaard, \emph{On the {H}ilbert-{B}lumenthal moduli problem}, J. Inst.
  Math. Jussieu \textbf{4} (2005), no.~4, 653--683.

\bibitem[Yu03]{yu:hilbmod}
C.-F. Yu, \emph{On reduction of {H}ilbert-{B}lumenthal varieties}, Ann. Inst.
  Fourier (Grenoble) \textbf{53} (2003), no.~7, 2105--2154.

\bibitem[Zar85]{zarhin:shafarevich}
Y.~Zarhin, \emph{A finiteness theorem for unpolarized abelian varieties over
  number fields with prescribed places of bad reduction}, Invent. Math.
  \textbf{79} (1985), no.~2, 309--321.

\end{thebibliography}
}

\vspace{5cm}

\vspace{0.1cm}
\noindent Rafael von K\"anel, IAS Tsinghua University, Beijing

\noindent E-mail address: {\sf rafaelvonkanel@gmail.com}

\vspace{1cm}

\noindent Arno Kret, University of Amsterdam, Amsterdam

\noindent E-mail address: {\sf  arnokret@gmail.com}

\end{document}